\documentclass{article}

\usepackage[T1]{fontenc}
\usepackage{lmodern}

\usepackage[nocompress]{cite}
\usepackage[hidelinks]{hyperref}
\usepackage{amsmath}
\usepackage{mathtools}
\usepackage{amssymb}
\usepackage{amsthm}
\usepackage{stmaryrd}
\usepackage{xcolor}
\usepackage{booktabs}
\usepackage{graphicx}
\usepackage{placeins}
\usepackage{enumitem}
\usepackage{microtype}
\usepackage{fullpage}
\usepackage{authblk}
\usepackage{aliascnt}
\usepackage{tikz}
\usetikzlibrary{arrows.meta,positioning,shapes.geometric}

\AddToHook{env/figure/begin}{\setlength{\abovecaptionskip}{4pt}}


\newtheorem{theorem}{Theorem}[section]

\newaliascnt{proposition}{theorem}
\newtheorem{proposition}[proposition]{Proposition}
\aliascntresetthe{proposition}

\newaliascnt{lemma}{theorem}
\newtheorem{lemma}[lemma]{Lemma}
\aliascntresetthe{lemma}

\newaliascnt{corollary}{theorem}

\aliascntresetthe{corollary}

\newaliascnt{definition}{theorem}

\aliascntresetthe{definition}

\newaliascnt{example}{theorem}

\aliascntresetthe{example}

\newaliascnt{remark}{theorem}

\aliascntresetthe{remark}

\makeatletter
\def\@listI{\leftmargin\leftmargini
  \parsep 4\p@ \@plus2\p@ \@minus\p@
  \topsep 4\p@ \@plus2\p@ \@minus2\p@
  \itemsep4\p@ \@plus2\p@ \@minus\p@}
\let\@listi\@listI
\@listi
\makeatother

\usepackage[nameinlink]{cleveref}
\crefname{theorem}{Theorem}{Theorems}
\Crefname{theorem}{Theorem}{Theorems}
\crefname{proposition}{Proposition}{Propositions}
\Crefname{proposition}{Proposition}{Propositions}
\crefname{lemma}{Lemma}{Lemmas}
\Crefname{lemma}{Lemma}{Lemmas}
\crefname{corollary}{Corollary}{Corollaries}
\Crefname{corollary}{Corollary}{Corollaries}
\crefname{definition}{Definition}{Definitions}
\Crefname{definition}{Definition}{Definitions}
\crefname{example}{Example}{Examples}
\Crefname{example}{Example}{Examples}
\crefname{remark}{Remark}{Remarks}
\Crefname{remark}{Remark}{Remarks}
\crefname{section}{Section}{Sections}
\Crefname{section}{Section}{Sections}
\crefname{table}{Table}{Tables}
\Crefname{table}{Table}{Tables}
\crefname{figure}{Figure}{Figures}
\Crefname{figure}{Figure}{Figures}

\newcommand{\Kfour}{K_4^{(3)}}
\newcommand{\A}{\mathcal A}
\newcommand{\F}{\mathcal F}
\newcommand{\Hom}{\operatorname{Hom}}
\newcommand{\PSD}{\succeq 0}
\newcommand{\bracket}[1]{\left\llbracket #1\right\rrbracket}
\newcommand{\R}{\mathbb R}
\newcommand{\Q}{\mathbb Q}
\newcommand{\N}{\mathbb N}

\newcommand{\ex}{\operatorname{ex}}
\newcommand{\Bexact}{\frac{14993367693127837}{26880000000000000}}

\title{A New Upper Bound for the Tur\'an Density of the Tetrahedron}
\date{\today}

\author[1,2]{Gyeongwon~Jeong\thanks{These authors are joint first authors of this work.}}
\author[1,2]{Seonghun~Park\protect\footnotemark[1]}
\author[3]{Seonghyuk~Im}
\author[2]{Joonkyung~Lee}
\author[1,2]{Hongseok~Yang}
\affil[1]{School of Computing, KAIST, Daejeon, Korea}
\affil[2]{School of Computational Sciences, Korea Institute for Advanced Study (KIAS), Seoul, Korea}
\affil[3]{Center for Artificial Intelligence and Natural Sciences, Korea Institute for Advanced Study (KIAS), Seoul, Korea}
\affil[ ]{E-mail addresses: \texttt{jgyw0910@kaist.ac.kr},
\texttt{hun57@kaist.ac.kr},
\texttt{seonghyuk@kias.re.kr},
\texttt{joonkyunglee@kias.re.kr},
\texttt{hongseokyang@kias.re.kr}}

\begin{document}
\maketitle

\begin{abstract}
We prove that the Tur\'an density of the tetrahedron $\Kfour$ satisfies
\[
  \pi(\Kfour)\;\le\;\Bexact\;<\;0.557789,
\]
improving Baber's upper bound of $0.5615$ and closing about $62\%$
of the gap to the conjectured value $5/9$.  The proof uses an exact
seven-vertex flag-algebra certificate incorporating degree-stationarity
from Razborov's differential method.
To find the certificate, we combine the established
techniques of cutting planes and column generation to optimize jointly
over flag families whose types have at most five vertices.
We give a complete formal proof of this Tur\'an density bound in
Lean~4.
\end{abstract}

\section{Introduction}\label{sec:intro}

Tur\'an's tetrahedron problem, posed in 1941 \cite{Turan1941}, asks
for the maximum number $\ex(n,\Kfour)$ of edges in a $3$-uniform
hypergraph on $n$ vertices containing no copy of the
\emph{tetrahedron} $\Kfour$, the complete $3$-uniform hypergraph on
four vertices.  Its asymptotic form asks for the Tur\'an density
\[
  \pi(\Kfour)\;=\;\lim_{n\to\infty}\frac{\ex(n,\Kfour)}{\binom n3},
\]
where the limit exists because a standard averaging argument shows
that the sequence $\ex(n,\Kfour)/\binom{n}{3}$ is non-increasing.
Tur\'an~\cite{Turan1941} constructed a family of
$\Kfour$-free $3$-graphs with asymptotic density $5/9$ and conjectured
that this is best possible, i.e., that $\pi(\Kfour)=5/9$.
More generally, Erd\H{o}s~\cite{Erdos1981} offered \$500
for determining the Tur\'an density of a complete $r$-uniform
hypergraph on $t$ vertices for any one pair satisfying $t>r>2$, and \$1000
for determining these densities for all such pairs.

Tur\'an's construction for the tetrahedron problem partitions the
vertices into three nearly equal parts $V_1,V_2,V_3$
and takes every triple meeting all three parts, together with every
triple having two vertices in $V_i$ and one in $V_{i+1}$, with indices
modulo $3$.  Brown \cite{Brown1983}, Kostochka \cite{Kostochka1982},
and Frohmader \cite{Frohmader2008} developed further constructions
attaining the conjectured density.  Fon-der-Flaass
\cite{FonderFlaass1988} gave a construction of $\Kfour$-free
$3$-graphs from oriented graphs with no induced directed four-cycle,
encompassing the constructions of Tur\'an, Brown, and Kostochka.
Razborov \cite{Razborov2011} proved that several classes of $3$-graphs
arising from the Fon-der-Flaass construction have asymptotic edge density
at most $5/9$, so a construction exceeding $5/9$ would have to lie outside
the classes covered by these results.

Related work has established the conjectured density under
additional restrictions on induced subgraphs
\cite{Razborov2010,Pikhurko2011,Razborov2014}.
Other variants concern the maximum sum of squared codegrees
\cite{BaloghClemenLidicky2022,BodnarEtAl2025} and the uniform
Tur\'an density \cite{KielakEtAl2026,Bucic2026}.
For further background on this problem, see Keevash
\cite{Keevash2011} and Balogh, Clemen and Lidick\'y
\cite{BaloghClemenLidickySurvey2022}.

Proving Tur\'an's conjecture requires the same asymptotic upper bound
of $5/9$ for arbitrary $\Kfour$-free $3$-graphs.
Following earlier bounds of de Caen \cite{deCaen1988} and
Giraud (unpublished; see \cite{ChungLu1999}), Chung and Lu
\cite{ChungLu1999} proved that
\[
  \pi(\Kfour)\le\frac{3+\sqrt{17}}{12}=0.593592\ldots.
\]
A substantial improvement came from Razborov's flag-algebra method
\cite{Razborov2007,Razborov2010}, which provides a systematic way
to derive inequalities among small subgraph densities using
semidefinite programming.  Applied to six-vertex $3$-graphs, this
method yielded the bound $0.561666$, with an exact verification
given by Baber and Talbot \cite{BaberTalbot2011}.
Baber \cite[\S5.1]{Baber2012} subsequently improved the bound
to $0.5615$ using partially defined, vertex-colored $3$-graphs
together with degree-regularity constraints.

We obtain the following upper bound, narrowing the gap to Tur\'an's
conjectured value.

\begin{theorem}[main bound]\label{thm:main}
The Tur\'an density of the tetrahedron satisfies
\[
  \pi(\Kfour)\;\le\;\Bexact\;<\;0.557789.
\]
\end{theorem}

The new bound reduces the gap between the upper bound and the
conjectured value $5/9$ from approximately $0.00594444$ to
$0.00223342$, closing about $62\%$ of the previous gap.

The proof of \cref{thm:main} rests on an exact seven-vertex
flag-algebra certificate.  The certificate combines flag-algebra
inequalities with a degree-stationarity identity from Razborov's
differential method \cite[\S4.3]{Razborov2007}, expressing the
asymptotic degree regularity of extremal configurations
\cite[\S5.1.1]{Baber2012}.
These ingredients alone, however, do not make the seven-vertex computation tractable.
The main computational idea is to avoid handling the resulting semidefinite program in full,
and instead generate only the variables and constraints needed as the optimization proceeds.

To find this certificate, we search over flag families for all $23$
tetrahedron-free five-vertex types up to relabeling, together with
selected smaller types.  This semidefinite program (SDP) involves positive
semidefinite matrices of size up to $1024\times1024$ and more than
a million coefficient inequalities.  Handling this scale is the main
computational challenge.

We address this challenge by combining the established techniques
of cutting planes \cite{Kelley1960} and column generation
\cite{GilmoreGomory1961,AhmadiDashHall2017}.
We restrict each matrix in the SDP to a nonnegative combination of
selected rank-one positive semidefinite matrices.
Positive semidefiniteness then holds automatically, and optimizing
the combination weights in place of the matrix entries gives a
linear program (LP).  We keep this LP small by starting with a few
rank-one matrices and enforcing only a subset of the coefficient
inequalities.
Column generation adds rank-one matrices that may improve the bound,
each with a new weight to optimize.  Cutting planes add omitted
coefficient inequalities violated by the current solution, so that
the next LP enforces them.  We solve the revised LP after each update.

This search uses floating-point arithmetic and produces numerical
matrices and coefficients for a candidate certificate.
To turn this candidate into a rigorous proof, we convert its data
to rational numbers and use exact arithmetic to recompute the bound
and verify positive semidefiniteness and all coefficient inequalities.
We also give a complete formal proof of the bound in Lean~4.

\paragraph{Organization.}
Readers interested only in the mathematical proof can focus on
\crefrange{sec:flag}{sec:cert}, which can be read independently of
the optimization and experiments in
\crefrange{sec:search}{sec:experiments}.
\Cref{sec:flag} recalls the flag-algebra framework and the
degree-stationarity property of edge-density maximizers.
\Cref{sec:cert} gives a criterion for certifying upper bounds and
proves \cref{thm:main} using an exact seven-vertex certificate.
\Cref{sec:search} describes the optimization method and its
implementation for the certificate search, while
\cref{sec:experiments} compares optimization methods and SDP
formulations through computational experiments.
The appendices provide coefficient calculations and exact arithmetic
checks of the certificate, implementation details for the search,
and experimental settings and additional results.

\section{Preliminaries}\label{sec:flag}

We recall the flag-algebra framework and the degree-stationarity
property needed for the proof of \cref{thm:main}.
Throughout the paper, $[n]=\{1,\dots,n\}$, and a \emph{$3$-graph}, short for
$3$-uniform hypergraph, is a pair $G=(V(G),E(G))$ with $V(G)$ finite
and $E(G)\subseteq\binom{V(G)}3$.  For $S\subseteq V(G)$, $G[S]$ is the
induced $3$-graph on $S$.  We use Razborov's flag-algebra framework
\cite{Razborov2007} with the universal theory $\mathcal{T}$ of tetrahedron-free
$3$-graphs.  Thus, all flags below are tetrahedron-free.  Since this
theory is fixed, we omit $\mathcal{T}$ from Razborov's notation.  For
instance, we write $\A^\sigma$ instead of $\A^\sigma[\mathcal{T}]$.

\subsection{Types, flags and densities}

A \emph{type} of size $k$ is a $\Kfour$-free $3$-graph $\sigma$ with
vertex set $[k]$.  The unique type of size $0$ is written $0$, and the
unique type of size $1$ is written $\sigma_1$, often shortened to $1$.
Up to
isomorphism, there is one type of each size $k\le2$, since no triple is
available; two of size $3$, as the triple $\{1,2,3\}$ is or is not an
edge; and four of size $4$, as zero, one, two or three of its four
triples are edges, all four being a tetrahedron.  \Cref{fig:types}
draws the eight types of size $1$ to $4$ and fixes the names by which
the certificate refers to them.

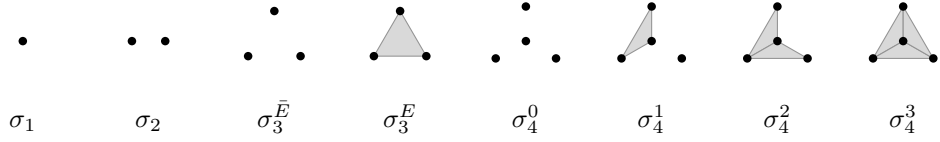
\begin{figure}[t]
\centering
\newcommand{\typepic}[1]{%
  \begin{tikzpicture}[scale=1]
    \useasboundingbox (-0.62,-0.60) rectangle (0.62,0.60);
    \tikzset{edg/.style={fill=gray!30, draw=gray!85, line join=round,
                         line width=0.4pt},
             vtx/.style={fill=black, draw=none}}
    #1
  \end{tikzpicture}}
\newcommand{\typeone}{\typepic{\fill[vtx] (0,0) circle (1.6pt);}}
\newcommand{\typetwo}{\typepic{%
  \fill[vtx] (-0.22,0) circle (1.6pt);
  \fill[vtx] ( 0.22,0) circle (1.6pt);}}
\newcommand{\threepts}{%
  \path (90:0.40) coordinate (a) (210:0.40) coordinate (b)
        (330:0.40) coordinate (c);}
\newcommand{\threedots}{%
  \fill[vtx] (a) circle (1.6pt); \fill[vtx] (b) circle (1.6pt);
  \fill[vtx] (c) circle (1.6pt);}
\newcommand{\typethreebar}{\typepic{\threepts\threedots}}
\newcommand{\typethreee}{\typepic{%
  \threepts \filldraw[edg] (a) -- (b) -- (c) -- cycle; \threedots}}
\newcommand{\fourpts}{%
  \path (0,0) coordinate (z) (90:0.46) coordinate (a)
        (210:0.46) coordinate (b) (330:0.46) coordinate (c);}
\newcommand{\fourdots}{%
  \fill[vtx] (z) circle (1.6pt); \fill[vtx] (a) circle (1.6pt);
  \fill[vtx] (b) circle (1.6pt); \fill[vtx] (c) circle (1.6pt);}
\newcommand{\typefour}[1]{\typepic{\fourpts #1 \fourdots}}
\newcommand{\sliceab}{\filldraw[edg] (z) -- (a) -- (b) -- cycle;}
\newcommand{\slicebc}{\filldraw[edg] (z) -- (b) -- (c) -- cycle;}
\newcommand{\sliceca}{\filldraw[edg] (z) -- (c) -- (a) -- cycle;}
\begin{tabular}{@{}cccccccc@{}}
\typeone & \typetwo & \typethreebar & \typethreee
 & \typefour{} & \typefour{\sliceab}
 & \typefour{\sliceab\slicebc} & \typefour{\sliceab\slicebc\sliceca}\\[2pt]
$\sigma_1$ & $\sigma_2$ & $\sigma_3^{\bar E}$ & $\sigma_3^{E}$
 & $\sigma_4^0$ & $\sigma_4^1$ & $\sigma_4^2$ & $\sigma_4^3$
\end{tabular}
\caption{The types of size $1$ to $4$ used to index the chosen flag
  blocks.  A shaded region is an edge, so $\sigma_3^E$ has its triple as
  an edge while $\sigma_3^{\bar E}$ does not, and $\sigma_4^j$ has $j$
  of its four triples as edges.}
\label{fig:types}
\end{figure}

Fix a type $\sigma$ of size $k$.  A \emph{$\sigma$-flag} is a pair
$F=(M,\theta)$ of a $\Kfour$-free $3$-graph $M$ and an injection
$\theta\colon[k]\hookrightarrow V(M)$ that is a $3$-graph isomorphism
from $\sigma$ onto $M[\theta([k])]$.  The vertices $\theta(i)$ are the
\emph{roots} (the labeled vertices) of $F$; the remaining, unlabeled
vertices are its \emph{non-root vertices}.  The \emph{size} of $F$
is $|V(M)|$, counting both roots and non-root vertices.  For
$S\supseteq\theta([k])$, $F[S]=(M[S],\theta)$ is the $\sigma$-flag
induced on $S$.  Two $\sigma$-flags
$(M_1,\theta_1)$ and $(M_2,\theta_2)$ are isomorphic if some $3$-graph
isomorphism $\alpha\colon M_1\to M_2$ satisfies
$\alpha\circ\theta_1=\theta_2$.  We write $\F^\sigma_\ell$ for the
finite set of $\sigma$-flags of size $\ell$ up to isomorphism,
$\F^\sigma=\bigcup_{\ell\ge k}\F^\sigma_\ell$, and
$1_\sigma=(\sigma,\mathrm{id}_{[k]})\in\F^\sigma_k$.
The $0$-flags are ordinary $\Kfour$-free
$3$-graphs, so $\F^0_\ell$ consists of these $3$-graphs on $\ell$
vertices up to isomorphism.  We write $\rho\in\F^0_3$ for the single
hyperedge.

Let $F\in\F^\sigma_m$ and $G=(M,\theta)\in\F^\sigma_n$ with $n\ge m$.
The \emph{density} $p(F;G)$ of $F$ in $G$ is the probability that, for
a uniformly random $(m-k)$-subset $S$ of $V(M)\setminus\theta([k])$,
the induced flag $G[S\cup\theta([k])]$ is isomorphic to $F$; for $n<m$,
put $p(F;G)=0$.  For $F_1\in\F^\sigma_{m_1}$ and $F_2\in\F^\sigma_{m_2}$
with $(m_1-k)+(m_2-k)\le n-k$, the \emph{pair density}
$p(F_1,F_2;G)$ is the probability that, for a uniformly random pair of
\emph{disjoint} subsets $S_1,S_2$ of $V(M)\setminus\theta([k])$ of
sizes $m_1-k$ and $m_2-k$, the flags $G[S_i\cup\theta([k])]$ are
isomorphic to $F_i$ for $i=1,2$; it is $0$ when
$(m_1-k)+(m_2-k)>n-k$.
Both densities are invariant under isomorphism of all their arguments.

\begin{lemma}[chain rule {\cite[Lemma~2.2]{Razborov2007}}]\label{lem:chain}
For all $F\in\F^\sigma_m$, $G\in\F^\sigma_n$ and $m\le\ell\le n$,
\[
  p(F;G)\;=\;\sum_{F'\in\F^\sigma_\ell}p(F;F')\,p(F';G).
\]
\end{lemma}

\subsection{The flag algebra}\label{sec:algebra}

Let $\R\F^\sigma$ be the real vector space with basis $\F^\sigma$, and
let $\mathcal K^\sigma\subseteq\R\F^\sigma$ be the subspace spanned by
the elements
\begin{equation}\label{eq:quotient-generators}
  F-\sum_{F'\in\F^\sigma_\ell}p(F;F')\,F',
  \qquad F\in\F^\sigma_m,\ \ell\ge m .
\end{equation}
The \emph{flag algebra} of $\sigma$ is the quotient space
$\A^\sigma=\R\F^\sigma/\mathcal K^\sigma$, in which the image of a
flag $F$ is written $F$ again.  Each expression in
\eqref{eq:quotient-generators} is zero in this quotient; these
equalities are the \emph{expansion relations}.
Thus, every $F\in\F^\sigma_m$ equals
its \emph{expansion} $\sum_{F'\in\F^\sigma_\ell}p(F;F')F'$ for every
$\ell\ge m$; these expansions are consistent by \cref{lem:chain}.
Expanding $1_\sigma$, whose density in every $\sigma$-flag is $1$,
gives
\begin{equation}\label{eq:unit}
  \sum_{F\in\F^\sigma_\ell}F\;=\;1_\sigma\qquad(\ell\ge k).
\end{equation}

The quotient carries a product as well.  For $F_1\in\F^\sigma_{m_1}$
and $F_2\in\F^\sigma_{m_2}$, set $\ell=m_1+m_2-k$, the number of
vertices occupied by two flags sharing their $k$ roots and having
disjoint non-root vertices.  Define their product by\footnote{We
use the smallest possible number of vertices to simplify the presentation.
For any $n\ge\ell$, the sum
$\sum_{F\in\F^\sigma_n}p(F_1,F_2;F)F$ gives the same element of
$\A^\sigma$, since the expressions in \eqref{eq:quotient-generators}
vanish in this quotient.}
\[
  F_1\cdot F_2\;:=\;\sum_{F\in\F^\sigma_\ell}p(F_1,F_2;F)\,F .
\]
Extend this formula bilinearly to $\R\F^\sigma$ and pass to the
quotient.  The result is independent of the representatives chosen
for the algebra elements.  It makes $\A^\sigma$ a commutative associative
algebra with unit $1_\sigma$ \cite[\S2]{Razborov2007}.
We write $1=1_0$ for the unit of $\A^0$.

The product reflects asymptotic independence.  For fixed
$\sigma$-flags $F_1,F_2$ and a $\sigma$-flag $G$ on $n$ vertices,
\[
  p(F_1,F_2;G)=p(F_1;G)p(F_2;G)+O(1/n)
  \qquad(n\to\infty).
\]
Choose the two sets of non-root vertices independently and uniformly,
with the sizes required for $F_1$ and $F_2$.  The probability of
obtaining both flags is $p(F_1;G)p(F_2;G)$.  Conditional on the sets
being disjoint, this probability is $p(F_1,F_2;G)$.
Their overlap probability is $O(1/n)$, since their sizes are fixed.
By the law of total probability, conditioning on disjointness changes
the probability of any event by at most this overlap probability,
giving the displayed estimate \cite[Lemma~2.3]{Razborov2007}.
As formalized in the next subsection, limiting densities preserve
the flag-algebra product.

\subsection{Positive homomorphisms}

A \emph{positive homomorphism} is an algebra homomorphism
$\phi\colon\A^\sigma\to\R$ (linear, multiplicative,
$\phi(1_\sigma)=1$) with $\phi(F)\ge0$ for every $F\in\F^\sigma$; the
set of positive homomorphisms is $\Hom^+(\A^\sigma,\R)$.  Applying
$\phi$ to \eqref{eq:unit} gives $\sum_{F\in\F^\sigma_\ell}\phi(F)=1$,
so each $\phi(F)$ lies in $[0,1]$.  A sequence $\{G_n\}$ of $\sigma$-flags with
$|V(G_n)|\to\infty$ is \emph{convergent} if
$\lim_{n\to\infty}p(F;G_n)$ exists for every $F\in\F^\sigma$.

\begin{theorem}[{\cite[Theorem~3.3]{Razborov2007}}]\label{thm:limits}
The elements of $\Hom^+(\A^\sigma,\R)$ are exactly the limits of
convergent sequences $\{G_n\}$ of $\sigma$-flags, with
$\phi(F)=\lim_{n\to\infty}p(F;G_n)$ for every $F\in\F^\sigma$.
\end{theorem}

For $f,g\in\A^\sigma$, we write $f\le g$ if $\phi(f)\le\phi(g)$ for
every $\phi\in\Hom^+(\A^\sigma,\R)$ \cite[\S3]{Razborov2007}, and
$f\le c$ for $f\le c\cdot1_\sigma$ when $c\in\R$.  Call $f$
\emph{nonnegative} if $f\ge0$.  Flags, squares, and nonnegative linear
combinations of nonnegative elements are nonnegative.

\begin{lemma}[coefficient bound]\label{lem:coeffbound}
If $f=\sum_{F\in\F^\sigma_\ell}c_F\,F$ in $\A^\sigma$, then
$f\le\max_{F\in\F^\sigma_\ell}c_F$.
\end{lemma}
\begin{proof}
For every $\phi\in\Hom^+(\A^\sigma,\R)$, positivity and
\eqref{eq:unit} give
\[
  \phi(f)=\sum_{F\in\F^\sigma_\ell}c_F\phi(F)
  \;\le\;\Bigl(\max_{F\in\F^\sigma_\ell}c_F\Bigr)
    \sum_{F\in\F^\sigma_\ell}\phi(F)
  \;=\;\max_{F\in\F^\sigma_\ell}c_F .\qedhere
\]
\end{proof}

The extremal problem can now be expressed in terms of positive
homomorphisms.

\begin{proposition}\label{prop:attained}
$\pi(\Kfour)=\max_{\phi\in\Hom^+(\A^0,\R)}\phi(\rho)$.
\end{proposition}
\begin{proof}
For each $n\ge3$, choose a $\Kfour$-free $3$-graph $G_n$ on $n$
vertices with $\ex(n,\Kfour)$ edges.  Since $\F^0$ is countable
and every density lies in $[0,1]$, a diagonal argument gives a
subsequence along which $p(F;G_n)$ converges for every $F\in\F^0$.
By \cref{thm:limits}, these limits define a positive homomorphism
$\psi$ with $\psi(\rho)=\pi(\Kfour)$.
Conversely, every $\phi\in\Hom^+(\A^0,\R)$ is realized by a
convergent sequence $\{G'_n\}$ with $|V(G'_n)|\to\infty$.
Taking limits in
\[
  p(\rho;G'_n)\le
  \frac{\ex(|V(G'_n)|,\Kfour)}{\binom{|V(G'_n)|}{3}}
\]
gives $\phi(\rho)\le\pi(\Kfour)$.
\end{proof}

\subsection{The downward operator}\label{sec:downward}

For a type $\sigma$ of size $k$ and a $\sigma$-flag $F=(M,\theta)$ of
size $m$, let $F|_0\in\F^0_m$ denote the underlying $3$-graph $M$,
obtained from $F$ by forgetting the root labeling $\theta$.  Let
$q_\sigma(F)\in(0,1]$ be the probability that a uniformly random
injection $\theta'\colon[k]\hookrightarrow V(M)$ makes $(M,\theta')$ a
$\sigma$-flag isomorphic to $F$.  The choice is over all injections,
without conditioning on the induced type.  Define
\[
  \bracket{F}_\sigma\;:=\;q_\sigma(F)\cdot F|_0 .
\]
The linear extension of this assignment to $\R\F^\sigma$ induces a map
$\bracket{\cdot}_\sigma\colon\A^\sigma\to\A^0$ on the quotient algebras
\cite[\S2.2]{Razborov2007}.  We call this map the \emph{downward}
(label-averaging) operator.

For example, let $H$ be the $3$-graph on four vertices with exactly
one edge.  Rooting $H$ at a vertex of the edge gives a $1$-flag
$F_{\mathrm{in}}$, whereas rooting it at the isolated vertex gives
a different $1$-flag $F_{\mathrm{out}}$.  Hence
\[
  \bracket{F_{\mathrm{in}}}_1=\frac34H,
  \qquad
  \bracket{F_{\mathrm{out}}}_1=\frac14H.
\]
The downward operator thus accounts for the probability of the
root placement as well as forgetting its label.

\begin{theorem}[{\cite[Theorem~3.1(a)]{Razborov2007}}]\label{thm:downward}
If $f\ge0$ in $\A^\sigma$, then $\bracket f_\sigma\ge0$ in $\A^0$.
\end{theorem}

For a vector of $\sigma$-flags $\mathbf F=(F_1,\dots,F_m)^{\mathsf T}$
and a positive semidefinite (PSD)\footnote{A
real symmetric matrix $Q\in\R^{m\times m}$ is positive semidefinite
if $x^{\mathsf T}Qx\ge0$ for every $x\in\R^m$.} matrix $Q\in\R^{m\times m}$,
choose a decomposition $Q=\sum_ru_ru_r^{\mathsf T}$.  Then, \cref{thm:downward}
gives
\[
  \bracket{\mathbf F^{\mathsf T}Q\mathbf F}_\sigma
  \;=\;\sum_r\bracket{(u_r^{\mathsf T}\mathbf F)^2}_\sigma\;\ge\;0 .
\]
Thus, PSD matrices give nonnegative quadratic forms in $\A^\sigma$,
and the downward operator maps these forms to nonnegative elements
of $\A^0$.

\subsection{Stationarity constraint}\label{sec:stationarity}

Let $d\in\F^1_3$ be the hyperedge with one root: the flag of type
$\sigma_1$ (shown in \cref{fig:types}) on three vertices whose
single edge contains the root and the two non-root vertices.
Every choice of the root gives
the same flag, so $q_{\sigma_1}(d)=1$ and
$\bracket{d}_1=\rho$, where $\rho$ is the unrooted three-vertex
$3$-graph consisting of a single edge.  We abbreviate a $1$-flag
$(G,\theta)$ to $(G,v)$,
where $v=\theta(1)$ is the root.
For a $1$-flag $(G,v)$ with $n=|V(G)|\ge3$,
\[
  p(d;(G,v))=\frac{\deg_G(v)}{\binom{n-1}{2}},
  \qquad \deg_G(v):=|\{e\in E(G):v\in e\}| .
\]
Consequently, if $\phi^1$ is realized by a convergent sequence
$\{(G_n,v_n)\}$ as in \cref{thm:limits}, then
$\phi^1(d)=\lim_{n\to\infty}\deg_{G_n}(v_n)/\binom{|V(G_n)|-1}{2}$,
the limit of the normalized degrees of the roots $v_n$ in $G_n$.
Define
\begin{equation}\label{eq:statelt}
  S\;:=\;3\bigl(\rho^2-\bracket{d\cdot d}_1\bigr)\;\in\;\A^0 .
\end{equation}
By \cite[Theorem~3.14]{Razborov2007} with $f=d$ and
$g=1_{\sigma_1}$, one has $S\le0$.  Call
$\phi\in\Hom^+(\A^0,\R)$ \emph{degree-stationary} when $\phi(S)=0$.
This condition means that the normalized degree of a uniformly
random vertex has vanishing variance along a convergent sequence
representing $\phi$.

\begin{proposition}\label{prop:maxstat}
Every maximizer of $\phi\mapsto\phi(\rho)$ on $\Hom^+(\A^0,\R)$ is
degree-stationary.
\end{proposition}
\begin{proof}
Let $\pi^1\colon\A^0\to\A^1$ and $\partial_1\colon\A^0\to\A^1$
be Razborov's upward operator and vertex differential, respectively
\cite[\S\S2.3.1 and~4.3]{Razborov2007}.  For $M\in\F^0_\ell$,
$\pi^1(M)$ is the sum of the $1$-flags on $\ell+1$ vertices whose
non-root vertices induce a copy of $M$.  For the hyperedge,
$\partial_1\rho=3(\pi^1(\rho)-d)$, so \cite[Theorem~2.8(a)]{Razborov2007}
gives
\[
  \bracket{d\cdot\partial_1\rho}_1
  =3\bigl(\rho\cdot\bracket{d}_1-\bracket{d\cdot d}_1\bigr)=S .
\]
Razborov calls a theory \emph{vertex uniform} if it has exactly one
one-vertex type \cite[Definition~11]{Razborov2007}.  This condition
holds for $\mathcal{T}$, since the only $3$-graph on a single vertex has no
edges.  We can therefore apply the stationarity criterion of Razborov
\cite[Corollary~4.6(a)]{Razborov2007}.  At an edge-density maximizer
$\phi_0$, take $M_1=\rho$, the objective $f(t)=-t$ on $\R$, and
$g=d$.  The choice $f(t)=-t$ makes $\phi_0$ a minimizer of the
objective in that corollary.  It gives
$\phi_0(\bracket{d\cdot\partial_1\rho}_1)=0$, hence $\phi_0(S)=0$.
\end{proof}

\section{Proof of \texorpdfstring{\Cref{thm:main}}{Theorem 1.1}}\label{sec:cert}

We first give a general criterion for certifying upper bounds on
the tetrahedron Tur\'an density.  We then present and verify an exact
seven-vertex certificate that proves \cref{thm:main}.

\subsection{Certificate criterion}\label{sec:bound}

We now introduce the notation needed to describe a certificate.
A \emph{block} $b$ is specified by a type $\sigma_b$ of size
$s_b$ and a list of distinct $\sigma_b$-flags on
$\ell_b\ge s_b$ vertices, up to isomorphism.
This coordinate list is the block's \emph{flag basis}, written as
$\mathbf F_b=(F_{b,1},\ldots,F_{b,d_b})^{\mathsf T}$.
Here $s_b$ is the number of roots.  We call the basis size $d_b$
the \emph{dimension} of block $b$; it is the number of rows
(equivalently, columns) of the associated matrix $Q_b$ below.

To each block $b$, we associate a positive semidefinite matrix
$Q_b\in\R^{d_b\times d_b}$ whose rows and columns are indexed
by its flag basis.  We write $Q_b\PSD$ to denote positive
semidefiniteness.  This matrix defines the quadratic form
$\mathbf F_b^{\mathsf T}Q_b\mathbf F_b\in\A^{\sigma_b}$.
Each product $F_{b,i}F_{b,j}$ appearing when this quadratic form
is multiplied out can be expanded using flags on $2\ell_b-s_b$
vertices by \cref{sec:algebra}.

Fix a real bound $B\in\R$.  Choose finitely many blocks with
their matrices $Q_b$, a real number $\tau$, and an \emph{expansion size}
$m\in\N$ with $2\ell_b-s_b\le m$ for every block $b$.
For convenience, we also take $m\ge6$.
The element $S$ in \eqref{eq:statelt} admits an expansion on six
vertices, and hence on $m$ vertices.
We can therefore expand all contributions using $3$-graphs on $m$ vertices.
For $G\in\F^0_m$, define the symmetric coefficient matrix $A_b(G)$
for block $b$ by
\begin{equation}\label{eq:coefficient-matrix}
  (A_b(G))_{ij}
  =\frac1{(m)_{s_b}}
    \sum_{\substack{\theta:[s_b]\hookrightarrow V(G)\\
                     \theta\text{ induces }\sigma_b}}
       p(F_{b,i},F_{b,j};(G,\theta)),
  \qquad (m)_s:=\frac{m!}{(m-s)!}.
\end{equation}
This averages over all injections of the roots, with zero
contribution from injections that do not induce $\sigma_b$.
By the definitions of the flag product and the downward operator,
$(A_b(G))_{ij}$ is the coefficient at $G$ in the expansion of
$\bracket{F_{b,i}F_{b,j}}_{\sigma_b}$ on $m$ vertices.
Likewise, the coefficient of $S$ is
\begin{equation}\label{eq:stationarity-coefficient}
  s(G)=3\left(p(\rho,\rho;G)
       -\frac1m\sum_{v\in V(G)}p(d,d;(G,v))\right).
\end{equation}
Writing $\langle Q,A\rangle=\sum_{i,j}Q_{ij}A_{ij}$, the full
coefficient at $G$ is therefore
\begin{equation}\label{eq:certificate-coefficient}
  c_G=\frac{|E(G)|}{\binom m3}
      +\sum_b\langle Q_b,A_b(G)\rangle+\tau s(G).
\end{equation}
The three terms are the contributions of the edge density,
the quadratic forms, and the stationarity constraint, respectively.
In particular, the coefficients are affine in the entries of
$Q_b$ and in $\tau$.

These expansions give the following identity in $\A^0$:
\begin{equation}\label{eq:certificate}
  \rho+\sum_b\bracket{\mathbf F_b^{\mathsf T}Q_b\mathbf F_b}_{\sigma_b}+\tau S
  =\sum_{G\in\F^0_m}c_G\,G.
\end{equation}
We call these data a \emph{certificate} for the bound
$\pi(\Kfour)\le B$ if $c_G\le B$ for every $G\in\F^0_m$.
The coefficient $\tau$ of $S$ is called the
\emph{stationarity multiplier}.

\begin{proposition}\label{prop:certificate-bound}
If $Q_b\PSD$ for every block and the coefficients in
\eqref{eq:certificate-coefficient} satisfy $c_G\le B$ for every
$G\in\F^0_m$, then $\pi(\Kfour)\le B$.
\end{proposition}
\begin{proof}
By \cref{prop:attained}, choose a maximizer $\phi$ with
$\phi(\rho)=\pi(\Kfour)$.  By \cref{prop:maxstat},
$\phi(\tau S)=0$ for every $\tau\in\R$, while each quadratic-form
contribution is nonnegative by \cref{thm:downward}.  Applying $\phi$ to
\eqref{eq:certificate} and using \cref{lem:coeffbound} gives
\[
  \pi(\Kfour)=\phi(\rho)
  \le\phi\!\left(\rho+
    \sum_b\bracket{\mathbf F_b^{\mathsf T}Q_b\mathbf F_b}_{\sigma_b}
    +\tau S\right)
  \le\max_{G\in\F^0_m}c_G\le B.\qedhere
\]
\end{proof}

We call such a certificate \emph{exact} when $B,\tau\in\Q$ and
each matrix is supplied as a finite sum
\[
  Q_b=\sum_{r=1}^{k_b}u_{b,r}u_{b,r}^{\mathsf T},
  \qquad u_{b,r}\in\Q^{d_b}.
\]
We call the vectors $u_{b,r}$ in this decomposition
\emph{factor vectors} of $Q_b$ and their number $k_b$ the
\emph{factor count} of the supplied decomposition.
This rational representation enables exact computer checks and
formal verification in Lean.

\subsection{The seven-vertex certificate}\label{sec:seven-certificate}

We now apply \cref{prop:certificate-bound} with expansion size $m=7$.
We specify the flag families and rational matrices, then verify
the resulting coefficient bound.
We use $23$ blocks:
one for $(s_b,\ell_b)=(1,4)$,
two each for $(3,5)$ and $(4,5)$, and $18$ for $(5,6)$.
For each block, the flag basis $\mathbf F_b$ consists of all
flags of its type and size.  The types, dimensions
$d_b$, and factor counts $k_b$, as defined in
\cref{sec:bound}, are summarized in \cref{tab:certificateblocks}.

\begin{table}[t]
\centering
\caption{Blocks of the exact certificate.
The types with at most four roots are shown in \cref{fig:types}.
The first column gives the number of vertices in each flag product.
The last row combines the $18$ blocks in \cref{tab:fivefamilies},
giving the ranges of their dimensions and factor counts.}
\label{tab:certificateblocks}
\small
\begin{tabular}{@{}ccccrr@{}}
\toprule
$2\ell_b-s_b$ & type size $s_b$ & flag size $\ell_b$ & type & dimension $d_b$ & factor count $k_b$\\
\midrule
6 & 4 & 5 & $\sigma_4^1$ & 56 & 1\\
6 & 4 & 5 & $\sigma_4^2$ & 50 & 2\\
7 & 1 & 4 & $\sigma_1$ & 7 & 7\\
7 & 3 & 5 & $\sigma_3^{\bar E}$ & 236 & 236\\
7 & 3 & 5 & $\sigma_3^{E}$ & 191 & 185\\
7 & 5 & 6 & five-root types ($18$ blocks) & $483$--$1024$ & $1$--$37$\\
\bottomrule
\end{tabular}
\end{table}

A six-vertex flag of a five-root type is obtained by adding one
non-root vertex $w$ and choosing which of the ten possible edges
$\{i,j,w\}$ to include, one for each pair of roots.
For every edge $\{i,j,k\}$ of the type, the three edges
$\{i,j,w\},\{i,k,w\},\{j,k,w\}$ cannot all be included,
since together with $\{i,j,k\}$ they would form a tetrahedron.
These are the only restrictions on the new edges.
Since an isomorphism must fix every root and the unique non-root
vertex, different allowed edge sets give different flags.
In particular, the empty five-root type has $2^{10}=1024$ flags.
The $18$ five-root types used in the certificate, together with
their dimensions and factor counts, are listed in
\cref{tab:fivefamilies} in Appendix~\ref{app:five-block-data}.

For each block $b$, the matrix $Q_b$ is specified by integer
vectors $a_{b,r}\in\mathbb Z^{d_b}$ with a common denominator
$M=8\cdot10^6$:
\[
  Q_b=\frac{1}{M^2}\sum_{r=1}^{k_b}a_{b,r}a_{b,r}^{\mathsf T}.
\]
Here $u_{b,r}=a_{b,r}/M$; we call the scaled vectors $a_{b,r}$
the \emph{integer factor vectors}.  The factor representation gives
a compact specification of the matrices and makes their positive
semidefiniteness explicit.  The total factor count is $\sum_b k_b=780$.
The integer factor vectors contain $326{,}557$ entries in total.
These vectors and the flag bases are supplied in the accompanying
data (see \nameref{sec:reproduction}).  The stationarity multiplier is
\[
  \tau=\frac{710104654129914}{64000000000000}\approx11.09539.
\]
Every flag product in \cref{tab:certificateblocks} uses at most
seven vertices, and $S$ can be expanded on any $m\ge6$ vertices.
For $G\in\F^0_7$, let $A_b(G)$ and $s(G)$ be given by
\eqref{eq:coefficient-matrix} and \eqref{eq:stationarity-coefficient}
with expansion size $m=7$.
Substituting the integer factor vectors into
\eqref{eq:certificate-coefficient} gives the explicit coefficient
\begin{equation}\label{eq:ansatz}
  c_G=\frac{|E(G)|}{35}
      +\frac1{M^2}\sum_b\sum_{r=1}^{k_b}
         a_{b,r}^{\mathsf T}A_b(G)a_{b,r}
      +\tau s(G).
\end{equation}
Thus, \eqref{eq:certificate} holds with $m=7$ and these coefficients.
Explicit integer counting formulas for these coefficients are
given in Appendix~\ref{sec:coeff}.
Their exhaustive verification is described in
Appendix~\ref{app:exact-verification}.
The following proposition gives the exact bound on every $c_G$.

\begin{proposition}\label{prop:seven-certificate}
For the blocks, integer factor vectors, and stationarity multiplier
specified above, set
\[
  L=5040M^2=322{,}560{,}000{,}000{,}000{,}000,
  \qquad C_7(G)=L\,c_G,
\]
where $c_G$ is given by \eqref{eq:ansatz}.
Then, $C_7(G)\in\mathbb Z$ for every $G\in\F^0_7$, and
\begin{equation}\label{eq:maxcol}
  \max_{G\in\F^0_7}C_7(G)=179{,}920{,}412{,}317{,}534{,}044.
\end{equation}
Consequently, these data form an exact certificate with
\[
  \max_{G\in\F^0_7}c_G=\Bexact=B.
\]
\end{proposition}
\begin{proof}
The factorization of each $Q_b$ is a sum of outer products of
rational vectors, so $Q_b\PSD$.
For the block sizes $(s_b,\ell_b)=(1,4),(3,5),(4,5),(5,6)$,
the average in \eqref{eq:coefficient-matrix} runs over, respectively,
$140$, $1260$, $5040$, and $5040$ equally weighted choices of
root-label assignments and ordered pairs of disjoint non-root sets.
All four numbers divide
$5040$.  The two averages in \eqref{eq:stationarity-coefficient}
have respectively $140$ and $630$ equally likely choices, which
also divide $5040$.  Since the edge-density denominator is $35$
and $M^2\tau$ is an integer, \eqref{eq:ansatz} implies
$C_7(G)\in\mathbb Z$.

To verify \eqref{eq:maxcol}, we take one representative of each of
the $964$ tetrahedron-free $3$-graphs on six vertices, add one
vertex, and consider all $2^{15}$ choices of edges containing it.
Retaining the tetrahedron-free extensions gives a list covering
every member of $\F^0_7$ up to isomorphism: deleting any vertex
from such a $3$-graph leaves a member of the six-vertex catalog.
Repeated isomorphism classes do not affect the maximum.
The exhaustive integer evaluations detailed in
Appendix~\ref{app:exact-verification} establish \eqref{eq:maxcol}
on this list; the counting formulas are given in
Appendix~\ref{sec:coeff}.
Dividing the maximum by $L$ gives $B=\Bexact$.
Together with the rational factorizations, this proves that the
data form the claimed exact certificate.
\end{proof}

Applying \cref{prop:certificate-bound} now proves \cref{thm:main}.

\paragraph{Formalization in Lean.}
We also use this certificate in a Lean formalization of the full
proof of \cref{thm:main}.  The formalization uses Lean~4
\cite{Lean4} and mathlib \cite{mathlib} and extends the
flag-algebra library for simple graphs developed in
\cite{FlagCompiler} to $3$-graphs.
The formal proof does not rely on
the results of the external verifiers and contains no omitted
proofs or additional user-declared axioms.  Most of the proofs
are checked directly by Lean's kernel, its core proof checker,
using Lean's standard axioms.  Finite computations, including
checks of the catalogs, symmetry tables, and inequalities
$c_G\le B$ for every $G\in\F^0_7$, use compiled Lean code for
efficiency.  These computational checks also require trust in Lean's
compiler.

\section{Optimization for certificate search}\label{sec:search}

This section explains how we found the certificate in
\cref{sec:cert}.  For fixed blocks, we seek positive
semidefinite matrices $Q_b$ and a stationarity multiplier $\tau$
that make the bound $B$ as
small as possible.  This leads to a semidefinite program involving
large matrices $Q_b$ and a coefficient inequality $c_G\le B$ for every tetrahedron-free
seven-vertex $3$-graph $G$.

Our algorithm solves a sequence of smaller linear
programs, using cutting planes to add constraints and column
generation to add variables as needed.  This process requires
repeated coefficient evaluations, which we accelerate by reusing
computations on six-vertex subgraphs and exploiting type symmetries.

\subsection{Semidefinite program}\label{sec:fullmodel}\label{sec:program}

We use expansion size $m=7$ and fix a family of blocks whose
flag products use at most seven vertices.
The coefficient matrices $A_b(G)$ and stationarity coefficients
$s(G)$ are those defined in \cref{sec:bound}.
Let $\mathcal{Q}=(Q_b)_b$ denote the collection of positive
semidefinite matrices, indexed by the blocks $b$.
We write $c_G(\mathcal{Q},\tau)$ for the coefficient in
\eqref{eq:certificate-coefficient} to make its dependence on the
optimization variables explicit.

We seek to minimize the largest of these coefficients.  Since
$c_G(\mathcal{Q},\tau)$ is affine in the matrix entries and $\tau$, this is
the semidefinite program
\begin{equation}\label{eq:sdp}
\begin{array}{ll}
\text{minimize} & B\\[2mm]
\text{subject to}
 & Q_b\PSD\quad(\text{every block }b),
   \qquad B,\tau\in\R,\\[1mm]
 & c_G(\mathcal{Q},\tau)\;\le\;B\qquad(G\in\F^0_7).
\end{array}
\end{equation}
The coefficient $\tau$ may be any real number, since the
stationarity condition $\phi(S)=0$ implies $\phi(\tau S)=0$
regardless of the sign of $\tau$.
Every feasible solution gives an upper bound on
$\pi(\Kfour)$ by \cref{prop:certificate-bound}; optimality is not
required.

For the search that produced the certificate in \cref{sec:cert},
we chose one representative of each type up to relabeling and
included a block for each representative of size $s$ and every
flag size $\ell>s$ satisfying $2\ell-s\in\{6,7\}$.
Each block uses all flags of its type and size as its
flag basis.  These $32$ blocks are listed in \cref{tab:fullfamilies}.
Only $23$ of these $32$ blocks have nonzero matrices in the final
certificate described in \cref{sec:cert}.

\begin{table}[t]
\centering
\caption{Blocks whose flag products use six or seven vertices,
with $\ell>s$.
The column $2\ell-s$ gives the number of vertices in each flag product.
Dimensions count flags up to isomorphism and
give the side lengths of the PSD matrices.}
\label{tab:fullfamilies}
\small
\begin{tabular}{@{}ccccc@{}}
\toprule
type size $s$ & flag size $\ell$ & $2\ell-s$ & type count & dimensions\\
\midrule
0 & 3 & 6 & 1 & 2\\
1 & 4 & 7 & 1 & 7\\
2 & 4 & 6 & 1 & 11\\
3 & 5 & 7 & 2 & $236,191$\\
4 & 5 & 6 & 4 & $64,56,50,45$\\
5 & 6 & 7 & 23 & $483$--$1024$\\
\bottomrule
\end{tabular}
\end{table}

\subsection{Cutting planes and column generation}\label{sec:implicit}

The SDP combines two computational challenges: PSD matrices of
size up to $1024\times1024$, whose independent entries are optimization
variables, and the inequalities $c_G(\mathcal{Q},\tau)\le B$ for all
$G\in\F^0_7$, where $|\F^0_7|=1{,}295{,}600$.
Directly optimizing over these large PSD blocks requires substantial
matrix computations, while including all these inequalities in the
optimization problem requires storing and processing a large amount
of coefficient data.

We apply the established techniques of cutting planes \cite{Kelley1960}
and column generation \cite{GilmoreGomory1961} to solve this SDP
through a sequence of smaller LPs.
To reduce the number of optimization variables, we express each $Q_b$ as
a nonnegative combination of selected matrices of the form
$vv^{\mathsf T}$, following the LP-based column-generation framework
of Ahmadi, Dash, and Hall \cite[\S3.1]{AhmadiDashHall2017}.
This guarantees positive semidefiniteness and allows the weights to
be optimized by linear programming.
We also enforce $c_G(\mathcal{Q},\tau)\le B$ only for $G$ in a subset of
$\F^0_7$ and gradually enlarge this subset as the search proceeds.
Our contribution here is the implementation of these techniques for
the seven-vertex flag-algebra SDP, with coefficient evaluation adapted
to tetrahedron-free $3$-graphs as described in \cref{sec:evaluation}.

More precisely, we choose a finite set of vectors
$V_b\subset\R^{d_b}$ for each block $b$ and a subset
$W\subseteq\F^0_7$.
Let $\lambda$ denote the collection of nonnegative weights
$\lambda_{b,v}$ over all blocks $b$ and vectors $v\in V_b$.
We parametrize each $Q_b$ as
\begin{equation}\label{eq:matrix-atoms}
 Q_b(\lambda)=\sum_{v\in V_b}\lambda_{b,v}vv^{\mathsf T},
 \qquad \lambda_{b,v}\ge0.
\end{equation}
This parametrization uses only the $|V_b|$ weights $\lambda_{b,v}$
as optimization variables for block $b$.
We write $\mathcal{Q}(\lambda)=(Q_b(\lambda))_b$ for the resulting
collection of matrices.

Combining the restrictions to $W$ and $V_b$ gives the
\emph{restricted LP} that we solve at each iteration:
\begin{equation}\label{eq:restricted-certificate}
\begin{array}{ll}
\text{minimize} & B\\[1mm]
\text{subject to}
 & \lambda_{b,v}\ge0\quad(v\in V_b,\ \text{every }b),
   \qquad B,\tau\in\R,\\[1mm]
 & c_G(\mathcal{Q}(\lambda),\tau)\le B\qquad(G\in W).
\end{array}
\end{equation}
The matrices $Q_b(\lambda)$ produced by the restricted LP are PSD by
\eqref{eq:matrix-atoms}.  However, the restricted LP solution may violate
$c_G(\mathcal{Q}(\lambda),\tau)\le B$ for some $G\in\F^0_7\setminus W$, so it need not
be feasible for the original SDP \eqref{eq:sdp}.
After each LP solve, we therefore evaluate the coefficients for all
tetrahedron-free seven-vertex $3$-graphs and compute
\[
 U=\max_{G\in\F^0_7}c_G(\mathcal{Q}(\lambda),\tau).
\]
Since $c_G(\mathcal{Q}(\lambda),\tau)\le U$ for every $G\in\F^0_7$,
replacing $B$ by $U$ gives a feasible solution
$(U,\mathcal{Q}(\lambda),\tau)$ to the original SDP in exact
arithmetic.  During the search, $U$ is evaluated in floating-point
arithmetic; the subsequent exact conversion produces the rigorous bound.

Our algorithm seeks to reduce $U$ by enlarging $W$ to enforce
more coefficient inequalities, or by enlarging the sets $V_b$ to
allow more choices for the matrices $Q_b$.
After each update, we solve the resulting restricted LP and
recompute $U$.  At termination, the algorithm returns the smallest
$U$ seen during the search together with the corresponding pair
$(\mathcal{Q},\tau)$.  The returned matrices and multiplier are
then converted into an exact certificate as described in
Appendix~\ref{app:exact-construction}.
The two strategies for enlarging $W$ and the sets $V_b$ are
described below; the full procedure is given in
Appendix~\ref{app:optimization-procedure}.

\paragraph{Adding violated constraints (cutting planes).}
We enlarge $W$ using a tolerance of $10^{-7}$:
if $U-B>10^{-7}$, we add selected $3$-graphs $G\notin W$ with
$c_G(\mathcal{Q}(\lambda),\tau)>B+10^{-7}$ to $W$, so that the corresponding
constraints are enforced in the next LP solve.

\paragraph{Adding vectors (column generation).}
We enlarge the sets $V_b$ to allow more choices for the matrices
$Q_b$.  Adding a vector $u\notin V_b$ to $V_b$ introduces a new
variable $\lambda_{b,u}\ge0$ and the corresponding matrix term
$\lambda_{b,u}uu^{\mathsf T}$ in \eqref{eq:matrix-atoms}.

To select the new vectors, we use additional information returned
by the LP solver. Along with the restricted LP solution, the solver
returns weights $y_G$, called \emph{dual multipliers}, associated
with the constraints $c_G(\mathcal{Q}(\lambda),\tau)\le B$ for $G\in W$.
For an optimal LP solution in exact arithmetic, these weights are
nonnegative and satisfy
$\sum_{G\in W}y_G=1$.
Multiplying the constraints by the weights $y_G$ and adding gives
\[
 \sum_{G\in W}y_Gc_G(\mathcal{Q}(\lambda),\tau)\le B.
\]
The solver chooses the weights $y_G$ so that the term involving
$\tau$ in this sum vanishes and every existing variable
$\lambda_{b,v}$ has a nonnegative coefficient.
Since $\lambda_{b,v}\ge0$, this implies
$B\ge\sum_{G\in W}y_G|E(G)|/35$.
This lower bound equals the optimal value of the restricted LP,
explaining why it cannot attain a smaller $B$ with its current
variables.
Appendix~\ref{app:dual-formulations} describes the dual LP used to
compute the weights $y_G$ and recover the certificate variables.

We look for vectors $u\notin V_b$ whose addition would introduce
the new variable $\lambda_{b,u}$ with a negative coefficient in the
weighted sum $\sum_{G\in W}y_Gc_G(\mathcal{Q}(\lambda),\tau)$.
Adding the matrix term $\lambda_{b,u}uu^{\mathsf T}$ adds
$\lambda_{b,u}u^{\mathsf T}A_b(G)u$ to each coefficient $c_G$.
Keeping the current weights $y_G$ fixed, the coefficient of the new
variable in this weighted sum is
\[
 \sum_{G\in W}y_G\,u^{\mathsf T}A_b(G)u
   =u^{\mathsf T}M_b(y)u,
 \qquad M_b(y)=\sum_{G\in W}y_GA_b(G).
\]
If this coefficient is negative, the same weights no longer
establish the current lower bound when the new variable is
included, so adding the term may allow a smaller $B$ after
reoptimization.

For each block, the minimum of $u^{\mathsf T}M_b(y)u$ over unit
vectors $u$ is the smallest eigenvalue of $M_b(y)$, attained by a
corresponding unit eigenvector.  We therefore compute the smallest
eigenvalues and corresponding unit eigenvectors of every $M_b(y)$.
Across all blocks, we select up to $80$ vectors corresponding to
the most negative eigenvalues, requiring each selected eigenvalue
to be below $-10^{-7}$.  We add each selected vector $u$ to its
corresponding set $V_b$ and solve the LP again, optimizing the
weights $\lambda_{b,u}$ of the new terms together with all existing
variables.

\subsection{Efficient coefficient evaluation}\label{sec:evaluation}

The search repeatedly evaluates $c_G(\mathcal{Q},\tau)$ for every
seven-vertex $3$-graph $G\in\F^0_7$.  The two reductions below
make these scans less expensive by reusing calculations across
$3$-graphs and across different labelings of the same roots.

\paragraph{Reusing six-vertex coefficients.}
The edge-density and stationarity terms, together with the six
blocks whose flag products use six vertices, can all be
expanded on six vertices.  For fixed $(\mathcal{Q},\tau)$, let $h_H$ be their
combined coefficient at $H\in\F^0_6$ in this expansion.
By the expansion relations \eqref{eq:quotient-generators}, their
combined contribution to $c_G(\mathcal{Q},\tau)$ is
\[
 \sum_{H\in\F^0_6}h_Hp(H,G)
 =\frac1{7}\sum_{v\in V(G)}h_{G-v},
\]
where $G-v$ is obtained by deleting vertex $v$ from $G$.
There are only $964$ $3$-graphs in $\F^0_6$, so we compute the
coefficients $h_H$ once for each $(\mathcal{Q},\tau)$ and reuse them
when evaluating $c_G$ for every $G\in\F^0_7$.
Evaluating this part of $c_G$ then requires just seven table lookups.

We use the same strategy of reusing six-vertex computations to
reduce the cost of forming $M_b(y)$ for these six blocks.
Let $A_b^{(6)}(H)$ be the matrix whose $(i,j)$ entry is the coefficient at
$H\in\F^0_6$ of $\bracket{F_{b,i}F_{b,j}}_{\sigma_b}$,
expanded on six vertices.  The same averaging applies entry by
entry to these coefficient matrices:
\[
 A_b(G)=\frac1{7}\sum_{u\in V(G)}A_b^{(6)}(G-u).
\]
We first combine the weights $y_G$ according to the six-vertex
$3$-graphs obtained by deletion.  For $H\in\F^0_6$, define
\[
 x_H=\frac1{7}\sum_{G\in W}y_G
       \bigl|\{u\in V(G):G-u\cong H\}\bigr|.
\]
Thus, $x_H$ is the total weight assigned to $H$ when each $y_G$
is distributed equally among the seven vertex deletions of $G$.
Regrouping the sum defining $M_b(y)$ by the resulting $3$-graph
$H$ gives
\[
 M_b(y)=\sum_{G\in W}y_GA_b(G)
       =\sum_{G \in W}y_G \cdot \frac1{7}\sum_{u\in V(G)}A_b^{(6)}(G-u)
       =\sum_{H\in\F^0_6}x_HA_b^{(6)}(H).
\]
We compute the $964$ scalar weights $x_H$ once for a given $y$
and reuse them for all six blocks.  Each $M_b(y)$ can then be
assembled from just $964$ six-vertex coefficient matrices.

\paragraph{Using type symmetries.}
For a five-root block $b$ of type $\sigma$, we use type symmetries
to evaluate its contribution $\langle Q_b,A_b(G)\rangle$ to $c_G$
more efficiently.
For $G\in\F^0_7$ and a root labeling $\theta:[5]\hookrightarrow V(G)$
that induces $\sigma$, let $x,y$ be the two vertices in
$V(G)\setminus\theta([5])$.  Let $i(\theta,x)$ and $i(\theta,y)$
be the indices in the flag basis $\mathbf F_b$ of block $b$ such that
\[
\begin{aligned}
 F_{b,i(\theta,x)}&\cong\bigl(G[\theta([5])\cup\{x\}],\theta\bigr),\\
 F_{b,i(\theta,y)}&\cong\bigl(G[\theta([5])\cup\{y\}],\theta\bigr).
\end{aligned}
\]
Expanding the block contribution using the flag product and
downward operator gives
\[
\begin{aligned}
 \langle Q_b,A_b(G)\rangle
 &=\frac1{2520}
 \sum_{\substack{\theta:[5]\hookrightarrow V(G)\\
                  \theta\text{ induces }\sigma}}
 \sum_{i,j}(Q_b)_{ij}\,p(F_{b,i},F_{b,j};(G,\theta))\\
 &=\frac1{2520}
 \sum_{\substack{\theta:[5]\hookrightarrow V(G)\\
                  \theta\text{ induces }\sigma}}
 \Bigl((Q_b)_{i(\theta,x),i(\theta,y)}\cdot\frac12
       +(Q_b)_{i(\theta,y),i(\theta,x)}\cdot\frac12\Bigr)\\
 &=\frac1{2520}
 \underbrace{
 \sum_{\substack{\theta:[5]\hookrightarrow V(G)\\
                  \theta\text{ induces }\sigma}}
 (Q_b)_{i(\theta,x),i(\theta,y)}
 }_{T_\sigma(G;Q_b)}.
\end{aligned}
\]
The first equality averages over all $2520=\binom75\,5!$
injections from $[5]$ to $V(G)$, with zero contribution from those
that do not induce $\sigma$.  The last equality uses the symmetry
of $Q_b$.  We write $T_\sigma(G;Q_b)$ for the final sum over labelings.

To evaluate $T_\sigma(G;Q_b)$, we precompute a matrix $\widehat Q_b$
from $Q_b$ that incorporates all automorphisms of the type $\sigma$.
Each $g\in\operatorname{Aut}(\sigma)$ permutes the flag basis of block $b$.
Write $g\cdot i$ for the index of the flag obtained by relabeling
$F_{b,i}$ according to $g$, and let $P_g$ be the permutation matrix
satisfying $P_ge_i=e_{g\cdot i}$, where $e_i$ is the $i$th coordinate
vector in $\R^{d_b}$.  Define
\[
 \widehat Q_b=\sum_{g\in\operatorname{Aut}(\sigma)}
                    P_g^{\mathsf T}Q_bP_g.
\]
This matrix has the same dimensions as $Q_b$ and depends only on
$Q_b$ and $\sigma$.

\begin{proposition}\label{prop:type-symmetry}
Let $G\in\F^0_7$.  For each five-vertex set $U\subseteq V(G)$
with $G[U]\cong\sigma$, choose a labeling
$\theta_U:[5]\hookrightarrow V(G)$ with image $U$ that induces
$\sigma$.  Let $\{x_U,y_U\}=V(G)\setminus U$ and set
$i_U=i(\theta_U,x_U)$ and $j_U=i(\theta_U,y_U)$.  Then,
\begin{equation}\label{eq:fiveorbit}
 T_\sigma(G;Q_b)=\sum_{\substack{U\subseteq V(G),\ |U|=5\\G[U]\cong\sigma}}
                  (\widehat Q_b)_{i_U,j_U}.
\end{equation}
Each summand is independent of the choice of $\theta_U$ and of
the ordering of $x_U,y_U$.
\end{proposition}
\begin{proof}
Fix one such set $U$.  Every labeling with image $U$ that induces
$\sigma$ differs from $\theta_U$ by a unique automorphism of
$\sigma$.  Hence
\[
 \sum_{\substack{\theta([5])=U\\\theta\text{ induces }\sigma}}
 (Q_b)_{i(\theta,x_U),i(\theta,y_U)}
 =\sum_{g\in\operatorname{Aut}(\sigma)}(Q_b)_{g\cdot i_U,g\cdot j_U}
 =(\widehat Q_b)_{i_U,j_U}.
\]
Summing over all such $U$ gives \eqref{eq:fiveorbit}.
Changing $\theta_U$ only permutes the automorphisms in the sum,
so the value is independent of the chosen labeling.  It is also
unchanged when $x_U$ and $y_U$ are interchanged, since $\widehat Q_b$
is symmetric.
\end{proof}

Once $\widehat Q_b$ is available, \cref{prop:type-symmetry} lets us
evaluate $T_\sigma(G;Q_b)$ for each $G\in\F^0_7$ by examining the
$\binom75=21$ five-vertex subsets $U\subseteq V(G)$, rather than
evaluating the contributions of up to $\binom75\,5!=2520$ injections
$\theta:[5]\hookrightarrow V(G)$ inducing $\sigma$ individually.
We reuse the same $\widehat Q_b$ for all $G\in\F^0_7$.

\section{Computational experiments}\label{sec:experiments}

We evaluated how effectively the algorithm of \cref{sec:search}
uses flag-algebra inequalities to obtain upper bounds on
$\pi(\Kfour)$ within a limited computation time.  The experiments
address two questions: which combination of optimization techniques
works well for these tetrahedron SDPs, and which additional flag
types improve the bounds obtained?  On two seven-vertex SDPs
tested, combining cutting planes and column generation produced
the strongest verified bound obtained from a one-hour search
among the four implementations.

Using cutting planes and column generation together, we compared
the bounds obtained with and without the five-root types.
Including all $23$ types improved the verified bound from
approximately $0.559210$ to $0.558582$ under a one-hour search budget.
A six-hour search with these types reached approximately
$0.557807$, and continuing that search until the convergence tests of
\cref{fig:pipeline} were met produced the bound in \cref{thm:main},
below $0.557789$ (Appendix~\ref{app:long-searches}).
These results show that optimizing over a broader range of flag
types, an established strategy in flag-algebra proofs, is effective
for the tetrahedron problem as well.  They also motivate
investigating our computational approach for other hypergraph
Tur\'an problems.

\subsection{Experimental design}\label{sec:experimentalsetup}

We compared four ways of handling the coefficient inequalities
and the PSD matrices.  The names follow \cref{sec:implicit}:
\texttt{CUT} denotes cutting planes, which add coefficient
inequalities by enlarging $W$, and \texttt{CG} denotes column
generation, which adds vectors to the sets $V_b$ and introduces
their weights as new LP variables.
\begin{itemize}[leftmargin=*,itemsep=3pt]
\item \texttt{SDP-FULL} included all coefficient inequalities
and optimized directly over PSD matrices using an SDP solver.
\item \texttt{LP-CG} included all coefficient inequalities and
used the representation \eqref{eq:matrix-atoms}, adding vectors
to the sets $V_b$ as the search proceeded.
\item \texttt{SDP-CUT} used an SDP solver for the PSD matrices
and generated coefficient inequalities by enlarging $W$.
\item \texttt{LP-CUT-CG} generated both the coefficient
inequalities and the vectors defining the matrices, as in
\cref{sec:implicit}.
\end{itemize}

For these comparisons, we used four flag-algebra optimization
problems for the tetrahedron.  These SDPs differ in the flag types and flag sizes
included and in the expansion size $m$.  All used the edge-density objective
and the degree-stationarity condition $\phi(S)=0$ from
\cref{sec:stationarity}.  The four SDPs were:
\begin{itemize}[leftmargin=*,itemsep=3pt]
\item \texttt{M6} used expansion size $m=6$ and the blocks
whose flag products use at most six vertices.
Its largest PSD matrix has size $64\times64$.
\item \texttt{M7-lift6} kept these blocks and increased the
expansion size to $m=7$, using the averaging formula in
\cref{sec:evaluation}.
\item \texttt{M7-no5} kept $m=7$ and added the block with
$(s,\ell)=(1,4)$ and the two blocks with $(s,\ell)=(3,5)$.
\item \texttt{M7-all5} also used $m=7$ and further added all $23$
five-root blocks with $(s,\ell)=(5,6)$.
\end{itemize}
The last SDP is the one used for the certificate search in
\cref{sec:fullmodel}.
Increasing the expansion size from $m=6$ to $m=7$ increases
the number of coefficient inequalities from $964$ to $1{,}295{,}600$.

We tested \texttt{M6}, \texttt{M7-no5}, and \texttt{M7-all5}
with each of the four optimization methods listed above.
For \texttt{M7-lift6}, we tested only \texttt{LP-CUT-CG}.
We ran each of the $13$ SDP--method combinations three times,
with a one-hour search budget per run.  Reported times and peak
memory measurements are medians of the three runs, with ranges
where shown.
Runs could terminate before the time limit if the numerical
stopping criteria were satisfied: the gap between the global
coefficient maximum and the returned dual objective, and all
violations of dual feasibility, had to be at most $10^{-7}$.
Appendix~\ref{app:experiment-details} gives the precise tests and
specifies the implemented block families, initialization, and
implementation settings.

All trials used an Intel Core i7-14700K with $63.72$~GiB RAM
and a $40$~GiB process-memory cap.  We used HiGHS
\cite{HuangfuHall2018} as the LP solver and SCS
\cite{ODonoghueEtAl2016} as the SDP solver.
Initialization and exact verification were excluded from the
search budget.

The common measure of progress was the smallest value of
$U=\max_{G\in\F^0_m}c_G(\mathcal Q,\tau)$ obtained by evaluating
every coefficient within the search budget, where $m$ is the
expansion size of the SDP.  This maximum gives an upper bound on
$\pi(\Kfour)$ whenever the matrices $Q_b$ are positive semidefinite
(\cref{prop:certificate-bound}), so smaller values indicate stronger
candidate bounds.
In the comparisons below, the plots show floating-point estimates
of $U$, and the tables report bounds obtained by converting the
selected candidates into exact certificates and checking every
coefficient inequality, as described in
Appendix~\ref{app:experiment-details}.  These verified bounds need
not be the optima of the corresponding SDPs.
Reproduction information is given in \nameref{sec:reproduction}.

\begin{table}[t]
\centering
\caption{Comparison of optimization methods on \texttt{M6} (expansion
size $m=6$) with a one-hour search budget per run.
Search times are reported as medians and ranges over three runs,
excluding initialization and exact verification.
Verified bounds are rounded to decimals; each method produced the
same verified bound in all three runs.
The total factor count is $\sum_b k_b$, summed over all blocks.
Both methods using column generation met the stopping criteria
in Appendix~\ref{app:experiment-details};
\texttt{SDP-FULL} reached the SDP solver's iteration limit, and
\texttt{SDP-CUT} reached the one-hour time limit.}
\label{tab:m6methods}
\small
\begin{tabular}{@{}lrrrr@{}}
\toprule
 & \multicolumn{2}{c}{search time (s)} & & \\
\cmidrule(lr){2-3}
method & median & range & verified bound & total factor count\\
\midrule
\texttt{SDP-FULL} & $2129.91$ & $[2089.44,2163.52]$ & $0.561665585236$ & $82$\\
\texttt{LP-CG} & $5.16$ & $[5.06,5.23]$ & $0.561670318397$ & $112$\\
\texttt{SDP-CUT} & $3595.47$ & $[3595.44,3595.47]$ & $0.572865345909$ & $176$\\
\texttt{LP-CUT-CG} & $6.91$ & $[6.89,7.20]$ & $0.561669244574$ & $109$\\
\bottomrule
\end{tabular}
\end{table}

\subsection{Comparison of optimization methods}\label{sec:algorithmcomparison}

We first held the SDP fixed and compared the four methods.
This tests how effectively each implementation searches
the same mathematical space of certificates within the available
time.  The SDP with $m=6$ provides a smaller reference problem for
understanding the results at $m=7$.

\subsubsection{Expansion size \texorpdfstring{$m=6$}{m=6}}\label{sec:m6comparison}

\Cref{tab:m6methods} summarizes the comparison of optimization
methods at expansion size $m=6$.
Both methods using column generation obtained bounds near
$0.56167$ and met the stopping criteria in
Appendix~\ref{app:experiment-details} before the one-hour time limit.
\texttt{LP-CG} took a median of $5.16$ seconds, compared with
$6.91$ seconds for \texttt{LP-CUT-CG}.  Using cutting planes
reduced the number of retained $3$-graphs from $964$ to $305$,
but increased the number of LP iterations from $58$ to
$139$.  Thus, on this small SDP, the reduced LP size did not
compensate for the additional iterations.

\texttt{SDP-FULL} produced a slightly stronger certificate,
approximately $0.5616656$, but took about $2130$ seconds and
reached the SDP solver's iteration limit before meeting the
requested accuracy.
\texttt{SDP-CUT} used almost the full hour but obtained a much
weaker bound.  These outcomes
show that the methods using column generation quickly obtained
bounds close to the best one found on \texttt{M6}; they also show
that cutting planes alone did not resolve the difficulties
encountered by the SDP solver in these runs.

\begin{table}[t]
\centering
\caption{Comparison of optimization methods on \texttt{M7-no5} and
\texttt{M7-all5} (expansion size $m=7$) with a one-hour search budget
per run.  Values are medians of three runs; verified bounds are
rounded to decimals.  All runs reached the one-hour time limit.
The search budget excludes initialization and exact verification.}
\label{tab:sevenmethods}
\small
\begin{tabular}{@{}llrr@{}}
\toprule
SDP & method & verified bound & peak memory usage (GiB)\\
\midrule
\texttt{M7-no5} & \texttt{SDP-FULL} & $0.615533579739$ & $24.17$\\
 & \texttt{LP-CG} & $0.561797610593$ & $15.05$\\
 & \texttt{SDP-CUT} & $0.568804336858$ & $1.16$\\
 & \texttt{LP-CUT-CG} & $0.559209644128$ & $0.99$\\
\addlinespace
\texttt{M7-all5} & \texttt{SDP-FULL} & $0.617248271542$ & $28.73$\\
 & \texttt{LP-CG} & $0.560454252478$ & $15.66$\\
 & \texttt{SDP-CUT} & $0.596480377283$ & $2.90$\\
 & \texttt{LP-CUT-CG} & $0.558581395366$ & $2.12$\\
\bottomrule
\end{tabular}
\end{table}

\begin{figure}[t]
\centering
\includegraphics[width=\textwidth,trim=0 6bp 0 0,clip]{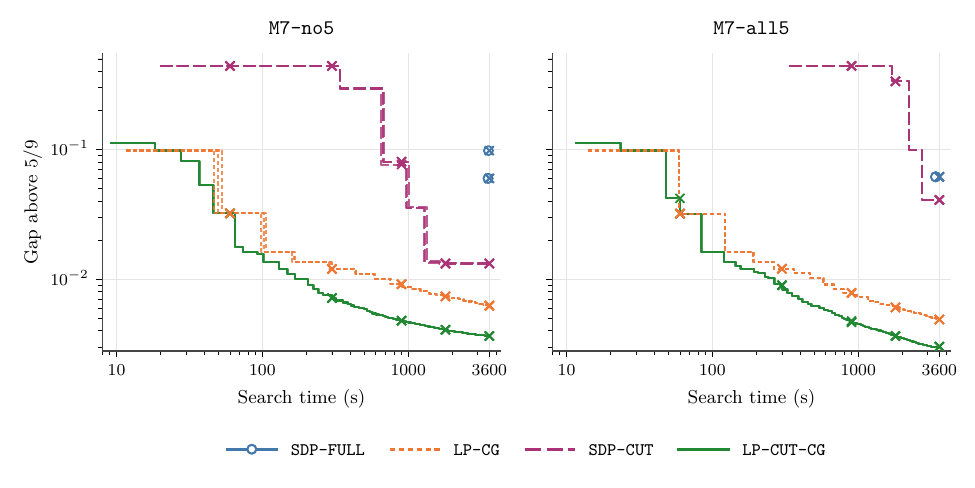}
\caption{Progress of the four methods on \texttt{M7-no5} and
\texttt{M7-all5} (expansion size $m=7$).
Each curve shows the smallest floating-point estimate of $U$ found
so far in one run, plotted as a gap above $5/9$.  Crosses mark
verified bounds for candidates available at the indicated search
budgets; verification time is additional.  Open circles mark the first
available \texttt{SDP-FULL} candidates.  Curves overlap when the
repetitions follow the same trajectory.}
\label{fig:sevenalgorithmconvergence}
\end{figure}

\subsubsection{Expansion size \texorpdfstring{$m=7$}{m=7}}\label{sec:no5algorithms}\label{sec:all5algorithms}

\Cref{tab:sevenmethods} summarizes the one-hour results of the
four optimization methods at expansion size $m=7$, and
\cref{fig:sevenalgorithmconvergence} shows their progress during
the search.
\texttt{SDP-FULL} gave the strongest bound on \texttt{M6}, but the
weakest bound on both \texttt{M7-no5} and \texttt{M7-all5}.
On these larger SDPs, \texttt{LP-CUT-CG}, which combines cutting
planes and column generation, gave the strongest bound within
the one-hour budget.

Cutting planes kept the LPs smaller, allowing more certificate
updates with lower peak memory usage.
On \texttt{M7-all5}, \texttt{LP-CUT-CG} completed about $253$
iterations in one hour, compared with $32$ for \texttt{LP-CG}.
\texttt{LP-CUT-CG} also used less than $1/7$ as much peak memory
as \texttt{LP-CG} ($2.12$ versus $15.66$~GiB).
This computational advantage was reflected in the verified
bound, which improved from approximately $0.560454$ with
\texttt{LP-CG} to $0.558581$ with \texttt{LP-CUT-CG}
(\cref{tab:sevenmethods}).
\Cref{fig:sevenalgorithmconvergence} shows that the advantage
also appeared during the search: the two methods initially
obtained similar bounds, but \texttt{LP-CUT-CG} subsequently
reached lower values of $U$ and maintained stronger bounds on
both SDPs.

The comparison of \texttt{LP-CUT-CG} with \texttt{SDP-CUT}
shows the benefit of column generation.
Both used cutting planes to reduce the number of coefficient
inequalities, but \texttt{LP-CUT-CG} obtained stronger verified
bounds on both \texttt{M7-no5} and \texttt{M7-all5}.
Thus, cutting planes alone were not enough to obtain comparably
strong bounds within the one-hour budget; adding column
generation led to substantially better bounds.

\begin{table}[t]
\centering
\caption{Comparison of the four SDP formulations under \texttt{LP-CUT-CG}
with a one-hour search budget per run.  Each SDP was run three times
and gave the same verified bound in all three runs.
Block counts refer to the implementations described in
Appendix~\ref{app:experiment-details}.
Verified bounds are rounded to decimals.  The total factor count is $\sum_b k_b$,
summed over all blocks.  The search budget
excludes initialization and exact verification.}
\label{tab:hostmodels}
\small
\begin{tabular}{@{}lrrrrr@{}}
\toprule
SDP & expansion size $m$ & block count & largest dimension & verified bound & total factor count\\
\midrule
\texttt{M6} & $6$ & $9$ & $64$ & $0.561669244574$ & $109$\\
\texttt{M7-lift6} & $7$ & $9$ & $64$ & $0.561671088720$ & $108$\\
\texttt{M7-no5} & $7$ & $12$ & $236$ & $0.559209644128$ & $447$\\
\texttt{M7-all5} & $7$ & $35$ & $1024$ & $0.558581395366$ & $511$\\
\bottomrule
\end{tabular}
\end{table}

\begin{figure}[t]
\centering
\includegraphics[width=\textwidth,trim=0 6bp 0 0,clip]{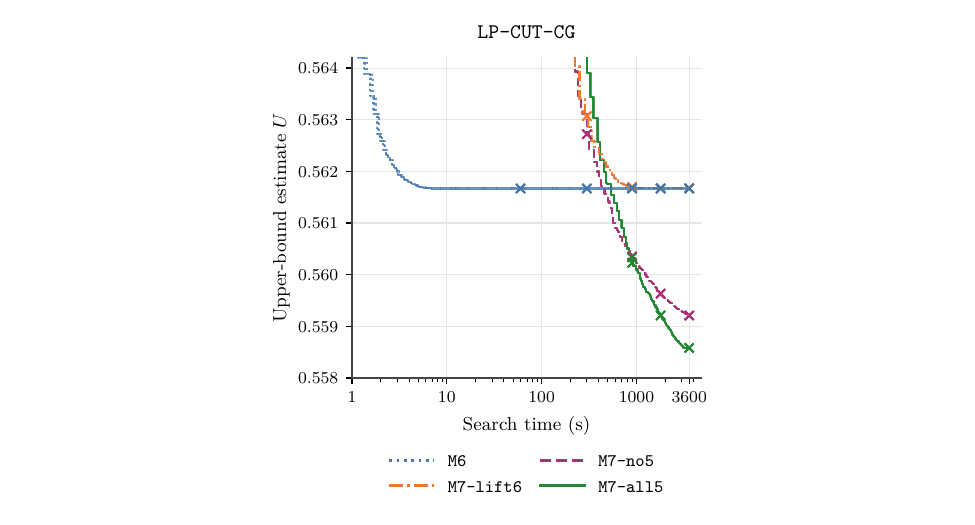}
\caption{Progress of the four SDPs with \texttt{LP-CUT-CG}.
Curves show the smallest
floating-point value of $U$ found so far in each run; crosses mark
exact certificate bounds at the indicated search budgets, excluding
verification time.  The plot focuses on the final bounds, with colors
and line styles distinguishing the four SDPs.  Results of
runs that stopped early are carried forward.}
\label{fig:remainingcomparison}
\end{figure}

\subsection{Comparison of SDP formulations}\label{sec:modeleffects}

We next fixed \texttt{LP-CUT-CG} and compared the four SDPs.
This compares the flag families and expansion sizes under the
same optimization method and one-hour budget.

\subsubsection{Expansion size and flag blocks}\label{sec:hostcomparison}

\Cref{tab:hostmodels} summarizes the one-hour results for the four
SDP formulations under \texttt{LP-CUT-CG}.
Increasing the expansion size from $m=6$ to $m=7$ while keeping
the same blocks did not improve the bound obtained.
\texttt{M7-lift6} met the stopping criteria in about $1188$ seconds,
with a verified bound approximately $1.84\cdot10^{-6}$ weaker
than that of \texttt{M6}.  The small difference is subject to
the stopping tolerances and rounding; it does not show that the
larger mathematical relaxation has a worse optimum.  It does
show that evaluating more coefficient inequalities, while
keeping these blocks fixed, brought no observed improvement
despite the increased computational cost.

In contrast, adding the one-root and three-root blocks whose
products use seven vertices improved the bound from approximately
$0.561671$ to $0.559210$.  Adding the five-root blocks improved
it further to $0.558581$.  Thus, the substantial gains in this
comparison came from allowing a broader family of PSD matrices
to contribute to the certificate.  \Cref{fig:remainingcomparison}
shows the progress of these searches.

\subsubsection{Effect of five-root blocks}\label{sec:freshcomparison}

The comparison between \texttt{M7-no5} and \texttt{M7-all5}
directly tests the benefit of optimizing jointly over all $23$
five-root types.  With the same initialization and budget, the
verified bound improved from approximately $0.559210$ to
$0.558581$, a decrease of approximately $0.00062825$.

This is evidence that the five-root blocks are useful in the
search for a stronger tetrahedron bound at the same expansion
size $m=7$.  Since neither SDP was
solved to optimality, the experiments do not establish a strict
difference between their optima.
A separate six-hour search with all $23$ five-root types reached
approximately $0.557807$; continued until the convergence tests were
met, it produced the certificate in \cref{sec:cert}.
Appendix~\ref{app:long-searches} gives the detailed results.

\paragraph{Acknowledgments.}
We thank Ingyu Baek, Jineon Baek, Bart{\l}omiej Kielak,
Taeyoung Kim, Daniel Kr\'a\v{l}, and Jaehyeon Seo for helpful 
discussions on flag algebras, hypergraph Tur\'an problems, and 
Lean formalization. 
S. Im was supported by a KIAS individual grant (AP109501) at the
Korea Institute for Advanced Study.
J. Lee was supported by Samsung STF Grant SSTF-BA2201-02 and 
the National Research Foundation of Korea (NRF) grant 
funded by the Korean Government (MSIT) (No.~NRF-2022R1C1C1010300).
G. Jeong, S. Park, and H. Yang were supported by another National
Research Foundation of Korea (NRF) grant funded by the Korean
Government (MSIT) (No.~RS-2023-00279680).

\paragraph{Data and code availability.}\label{sec:reproduction}

The accompanying repository\footnote{\url{https://github.com/taeyool/tetrahedron-turan}}
contains the search code and experimental records, rational
certificates, independent integer verifiers, and Lean sources.
The computational environment and verification tools are
described in \path{docs/reproduction.md}.
For \cref{thm:main}, \path{docs/lean-proof.md}
identifies the exact certificate and gives instructions for
reproducing the Lean proof, including the required software
versions and the recorded verification results.
The certificate can also be checked independently of the
optimization that produced it.

\paragraph{AI disclosure.}
AI coding agents assisted implementation, proof development, and
the writing of this paper.  Claude Code helped extend the Lean
formalization of flag algebras \cite{FlagCompiler} from finite
simple graphs to $3$-uniform hypergraphs.  Codex assisted the
optimization code, experimental evaluation, and formalization
of the certificate.

\bibliographystyle{unsrt}
\bibliography{references}

\newpage
\appendix
\section{Certificate data and exact verification}\label{app:certificate-details}

This appendix describes the five-root block data, derives the
integer formulas for the certificate coefficients, and explains
how the coefficient bound in \cref{prop:seven-certificate} is
verified exhaustively.

\subsection{Five-root block data}\label{app:five-block-data}

Up to isomorphism, the five-root types are the $23$
tetrahedron-free $3$-graphs on five vertices.  For each
isomorphism class, we choose a
representative with vertex set $[5]$ as follows.  List the ten
triples of $[5]$ lexicographically as $e_0,\ldots,e_9$.
For each labeling of the vertices by $[5]$, encode the edge set
$E$ by the integer $\sum_{j:e_j\in E}2^j$.
The \emph{type identifier} $\iota$ is the smallest such integer,
and the representative is the $3$-graph on $[5]$ that it encodes.
\cref{tab:fivefamilies} and the certificate data refer to each
type by its identifier.  The verification code, including the Lean
formalization, uses these identifiers to reconstruct the edge
sets of the types.
The matrix $Q_b$ is nonzero for exactly $18$ of the $23$ types in
the certificate of \cref{sec:cert}.  \cref{tab:fivefamilies} lists
these $18$ types.
For each listed type, the dimension $d_b$ is the number of its
six-vertex flags, and the factor count $k_b$ is the number of
factor vectors in the supplied decomposition of $Q_b$.

\begin{table}[t]
\centering
\caption{The $18$ five-root types with nonzero matrices in the
certificate, listed by type identifier.  The factor counts total
$349$.}
\label{tab:fivefamilies}
\small
\renewcommand{\arraystretch}{1.15}
\setlength{\tabcolsep}{16pt}
\begin{tabular}{@{}rrr@{}}
\toprule
type identifier $\iota$ & dimension $d_b$ & factor count $k_b$\\
\midrule
0 & 1024 & 21\\
1 & 896 & 30\\
3 & 800 & 35\\
7 & 728 & 12\\
11 & 720 & 23\\
15 & 652 & 23\\
30 & 644 & 34\\
31 & 594 & 37\\
63 & 545 & 2\\
77 & 640 & 1\\
87 & 584 & 6\\
94 & 578 & 28\\
116 & 637 & 14\\
117 & 580 & 17\\
119 & 532 & 31\\
222 & 526 & 10\\
237 & 527 & 17\\
254 & 483 & 8\\
\bottomrule
\end{tabular}
\end{table}

\subsection{Integer coefficient formulas}\label{sec:coeff}

We express the coefficients $c_G$ in \eqref{eq:ansatz}
in terms of the integer counts computed by the verification code.
For each count, the formulas below identify the flag-algebra term
that it represents and state its normalizing denominator, so that
the computation can be checked against the mathematical definitions.

To check the bound in integer arithmetic, we clear the common
denominator $L=5040M^2$ and compute the integer numerator
$C_7(G)=L\,c_G$ defined in \cref{prop:seven-certificate}.
We evaluate $C_7(G)$ in two parts.  First, as in
\cref{sec:evaluation}, we expand the edge density, the two
four-root blocks, and the stationarity term on six vertices,
computing their coefficients once for each $H\in\F^0_6$ and
reusing them for every seven-vertex $3$-graph $G\in\F^0_7$.
Second, we evaluate the remaining blocks, whose flag products
use seven vertices, directly on $G$.

For the first part, define $C_6(H)$ for each $H\in\F^0_6$ by the
following expansion over the common denominator $720M^2$:
\[
  \rho+\sum_{b:\,s_b=4}
    \bracket{\mathbf F_b^{\mathsf T}Q_b\mathbf F_b}_{\sigma_b}+\tau S
  =\sum_{H\in\F^0_6}\frac{C_6(H)}{720M^2}\,H.
\]
To compute the coefficients of the two four-root blocks, fix
one such block $b$ and assign its four root labels to distinct
vertices of $H$.  The two remaining vertices complete the
first and second five-vertex flags, one vertex for each flag.
The root labels can be assigned in $6\cdot5\cdot4\cdot3$ ways
and the remaining vertices in two ways, giving $720$
choices in total.  The \emph{count matrix} $N_b(H)$ has $(i,j)$
entry equal to the number of choices whose first and second flags
are isomorphic to $F_{b,i}$ and $F_{b,j}$.
Choices whose roots do not induce $\sigma_b$ are not counted.
Dividing by the $720$ choices gives $N_b(H)/720=A_b^{(6)}(H)$,
the six-vertex coefficient matrix of \cref{sec:evaluation}.

The $r$th factor vector $u_{b,r}=a_{b,r}/M$ weights each flag pair
$F_{b,i},F_{b,j}$ by
$a_{b,r,i}a_{b,r,j}/M^2$, where $a_{b,r,i}$ is the $i$th
coordinate of $a_{b,r}$.  The coefficient of its term at $H$ is
therefore
\[
  \frac{1}{720M^2}\sum_{i,j}a_{b,r,i}a_{b,r,j}N_b(H)_{ij}
  =\frac{a_{b,r}^{\mathsf T}N_b(H)a_{b,r}}{720M^2}.
\]
Summing over $r=1,\ldots,k_b$ gives the coefficient of block $b$ at $H$.

We next compute the stationarity coefficient by expanding
the two terms of $S$ in \eqref{eq:statelt}, $\rho^2$ and
$\bracket{d\cdot d}_1$, separately on six vertices.
For $\rho^2$, let $n_1(H)$ count the choices of an ordered pair
of disjoint edges and a distinguished vertex in the first edge.
The expansion of $\rho^2$ on six vertices averages over the
$20$ ordered pairs of complementary triples, of which
$n_1(H)/3$ are pairs of disjoint edges.
For $\bracket{d\cdot d}_1$, let $n_2(H)$ count the ordered pairs
of edges intersecting in exactly one vertex; their intersection
determines the root uniquely.
The expansion of $\bracket{d\cdot d}_1$ averages over
the $6\binom52\binom32=180$ choices of a root and an ordered pair
of disjoint $2$-subsets of the remaining vertices, of which
$n_2(H)$ have both $2$-subsets forming edges with the root.  The
coefficients of
$\rho^2$ and $\bracket{d\cdot d}_1$ at $H$ are therefore
$n_1(H)/60$ and $n_2(H)/180$, respectively.  Substituting them
into \eqref{eq:statelt} gives
\begin{equation}\label{eq:statcoeff}
  S\;=\;\sum_{H\in\F^0_6}\frac{3n_1(H)-n_2(H)}{60}\,H .
\end{equation}
The verification code computes the integers $3n_1(H)-n_2(H)$
directly from $H$.

The edge density $\rho$ has coefficient $|E(H)|/20$ at $H$, the
fraction of the $20$ triples of $V(H)$ that are edges.
Write $\tau=\tau_{\mathrm{num}}/M^2$ with
$\tau_{\mathrm{num}}=710104654129914$.
Adding the edge-density coefficient, the two block
coefficients, and $\tau$ times the coefficient in
\eqref{eq:statcoeff}, then multiplying by $720M^2$, gives
\[
  C_6(H)=36\,|E(H)|M^2
       +\left(\sum_{b:\,\sigma_b\in\{\sigma_4^1,\sigma_4^2\}}
          \sum_r a_{b,r}^{\mathsf T}N_b(H)a_{b,r}\right)
       +12\,\tau_{\mathrm{num}}\bigl(3n_1(H)-n_2(H)\bigr).
\]
The constants $36=720/20$ and $12=720/60$ come from clearing the
denominators $20$ and $60$.

By the expansion relations \eqref{eq:quotient-generators},
the combined coefficient at $G\in\F^0_7$ of the edge density, the
two four-root blocks, and the stationarity term is the average
$\frac17\sum_{v\in V(G)}C_6(G-v)/(720M^2)$ of their six-vertex
coefficients over the seven $3$-graphs $G-v$ obtained by deleting
one vertex $v$.  Multiplying this average
by $L=7\cdot720M^2$ gives the term
$\sum_{v\in V(G)}C_6(G-v)$ in the integer formula for $C_7(G)$.

For the second part,
we evaluate the remaining one-root, three-root, and five-root
blocks directly on $G$.
For each of these blocks $b$, we use the integer matrix
$\sum_r a_{b,r}a_{b,r}^{\mathsf T}=M^2Q_b$ in place of $Q_b$.
Combining the integer factor vectors into one matrix lets us write the sum
of their individual weighted counts as a single matrix inner product.
We denote this integer matrix by $Q_1,Q_{\bar E},Q_E$ for the types
$\sigma_1,\sigma_3^{\bar E},\sigma_3^E$, and by $Q_\sigma$
for each five-root type $\sigma$.

For the one-root and three-root blocks, define the count matrices
$N_1(G),N_{\bar E}(G),N_E(G)$ as for the four-root blocks:
assign the root labels, then split the remaining vertices into
a first set and a second set of equal size to complete the two flags.
For one root there are $7\binom63=140$ choices; for three roots
there are $7\cdot6\cdot5\binom42=1260$ choices.
Dividing each count matrix by its number of choices gives the
matrix $A_b(G)$ of \eqref{eq:coefficient-matrix} with $m=7$.
With $\langle\cdot,\cdot\rangle$ as in \cref{sec:bound}, the
coefficients of these three blocks at $G$ are therefore
$\langle Q_1,N_1(G)\rangle/(140M^2)$,
$\langle Q_{\bar E},N_{\bar E}(G)\rangle/(1260M^2)$, and
$\langle Q_E,N_E(G)\rangle/(1260M^2)$, respectively.
Multiplying by $L$ gives the constants $36=5040/140$ and
$4=5040/1260$ for these blocks.

For a five-root block $b$ of type $\sigma$, \cref{sec:evaluation}
shows that $\langle Q_b,A_b(G)\rangle=T_\sigma(G;Q_b)/2520$.
Since $T_\sigma$ is linear in the matrix, the coefficient of the
block at $G$ is $T_\sigma(G;Q_\sigma)/(2520M^2)$,
so multiplying by $L$ gives $2T_\sigma(G;Q_\sigma)$.
Adding these terms to the previously computed
$\sum_{v\in V(G)}C_6(G-v)$ gives the full integer coefficient
\begin{equation}\label{eq:column}
 \begin{aligned}
  C_7(G)\;={}&\sum_{v\in V(G)}C_6(G-v)
  +36\,\langle Q_1,N_1(G)\rangle\\
  &+4\,\langle Q_{\bar E},N_{\bar E}(G)\rangle
  +4\,\langle Q_E,N_E(G)\rangle
  +2\sum_\sigma T_\sigma(G;Q_\sigma).
 \end{aligned}
\end{equation}
The last sum runs over the $18$ five-root types in
\cref{tab:fivefamilies}.
Thus, the right-hand side of \eqref{eq:column} equals $L\,c_G$
with $c_G$ as in \eqref{eq:ansatz}, so the integer comparison
$C_7(G)\le LB$ is equivalent to $c_G\le B$.

\subsection{Exhaustive verification}\label{app:exact-verification}

We verify $C_7(G)\le LB$ for every $G\in\F^0_7$ by evaluating the
integer formula \eqref{eq:column}, with $L$ and $B$ as in
\cref{prop:seven-certificate}.
Before evaluating the coefficients, we check that each integer
factor vector $a_{b,r}$ has $d_b$ entries, indexed by the supplied flag
basis $\mathbf F_b$.  We also bound the intermediate and final
integer values to ensure that the arithmetic does not overflow.
Below, we describe how the enumeration covers every
$G\in\F^0_7$, how two implementations outside Lean cross-check the
coefficient calculations, and how the Lean formalization
establishes the same bound independently.

To cover all of $\F^0_7$, we construct seven-vertex $3$-graphs by
extending tetrahedron-free $3$-graphs from a complete six-vertex catalog.
The catalog contains one representative of each of the $2136$
isomorphism classes of six-vertex $3$-graphs, of which $964$ are
tetrahedron-free.  For each of these $964$ representatives, we add
one vertex and consider all $2^{15}$ choices of edges containing
it.  Of these extensions, $13{,}051{,}375$ are tetrahedron-free.
Every member of $\F^0_7$ occurs among them up to isomorphism,
because deleting any vertex of such a $3$-graph leaves a
tetrahedron-free six-vertex $3$-graph, which is isomorphic to one
of the catalog representatives.  Isomorphic repetitions do not affect
the maximum.

Two implementations evaluate \eqref{eq:column} on all these
extensions and cross-check each other.
One reuses the search's coefficient evaluator in exact arithmetic
mode; the other is a separate C\texttt{++} verifier that reconstructs
the five-root catalog independently.
Both implementations take the six-vertex values $C_6(H)$ from one
table, which the verification script computes from the integer factor
vectors by the formulas of Appendix~\ref{sec:coeff}; the two scans
therefore cross-check only the seven-vertex terms.
For one-root and three-root
blocks, the C\texttt{++} verifier enumerates the root-label
assignments of
Appendix~\ref{sec:coeff}; for five-root blocks, it evaluates
$2T_\sigma(G;Q_\sigma)$ using the \emph{grouped} sum
\eqref{eq:fiveorbit} over five-vertex root sets.
The two implementations agree on the maximizing $3$-graph
and give the exact maximum in \eqref{eq:maxcol}, thereby
establishing the coefficient bound in \cref{prop:seven-certificate}.
For the maximizing $3$-graph and selected others, a Python script
additionally enumerates all root-label assignments and flag pairs
for the one-, three-, and five-root blocks and confirms the
component values, including agreement with the grouped sum.

The Lean formalization of \cref{thm:main} described in
\cref{sec:cert} establishes the coefficient bound
independently of these implementations.  It proves that
the catalog contains every tetrahedron-free six-vertex $3$-graph up
to isomorphism by examining all $2^{20}$ edge sets on $[6]$, and
proves that its extensions cover every
tetrahedron-free seven-vertex $3$-graph up to isomorphism.
It checks the coefficient inequalities by exhaustive computation
within Lean and identifies the computed values with the coefficients
of the corresponding flag-algebra expressions.  For the five-root
blocks, it proves that the grouped sum and the sum
over root labelings yield the same integer coefficient
$2T_\sigma(G;Q_\sigma)$ in \eqref{eq:column}.
These finite computations use compiled Lean code, as discussed
in \cref{sec:cert}.

\section{Optimization procedure and implementation}\label{app:search-implementation}

This appendix gives the algorithmic and implementation details
for the certificate search in \cref{sec:search}.

\subsection{Optimization procedure}\label{app:optimization-procedure}

\Cref{fig:pipeline} gives the procedure used in the search leading
to our final bound, combining the cutting-plane and column-generation
steps of \cref{sec:implicit}.  We initialize $W$ with all seven-vertex
induced subgraphs occurring in Tur\'an's balanced cyclic three-part
construction from \cref{sec:intro}.
For every block $b$, we initially set $V_b=\varnothing$.

After each LP solve, we scan the
coefficients for violated constraints, as in the cutting-plane step,
and compute eigenvectors of $M_b(y)$ with negative eigenvalues,
as in the column-generation step.  The scan evaluates every
$G\in\F^0_7$ by enumerating all tetrahedron-free one-vertex
extensions of each $H\in\F^0_6$.  For each $H$, we retain up to
$8$ extensions with the largest coefficients and select as
candidates those with $c_G(\mathcal{Q},\tau)>B+10^{-7}$.
These candidates may include graphs already in $W$,
whose constraints can appear violated because of floating-point errors.
Among the selected candidates, we add only the graphs $G\notin W$
to $W$, enforcing their constraints in the next LP.
If no new constraint is added, we instead add the selected vectors
to the corresponding sets $V_b$.
This includes the case where all selected constraints are already
represented in $W$ and $U-B>10^{-7}$, so that the next LP can use new
matrix terms instead of simply repeating the same problem.
This does not relax the convergence criteria in \cref{fig:pipeline}.

To control the LP size, we periodically remove vectors with
negligible weights $\lambda_{b,v}$ from the corresponding sets $V_b$.
Before declaring convergence, we also check the LP solution's
feasibility: the errors in $\sum_Gy_G=1$ and $\sum_Gy_Gs(G)=0$,
and any violation of $y_G\ge0$, must each be at most $10^{-7}$.

\begin{figure}[t]
\centering
\setlength{\fboxsep}{6pt}
\fbox{\begin{minipage}{0.94\linewidth}
\small
Initialize $W$ with all seven-vertex
induced subgraphs occurring in Tur\'an's balanced cyclic three-part
construction and set
$V_b=\varnothing$ for every block.
Repeat:
\begin{enumerate}[label=\arabic*.,leftmargin=*,itemsep=3pt,topsep=4pt]
\item \textbf{Solve the restricted LP.} Obtain its value $B$,
the matrices $Q_b$, the multiplier $\tau$, and the dual multipliers
$y_G$ for $G\in W$.
\item \textbf{Select vectors.} Compute the smallest eigenvalues
and corresponding unit eigenvectors of every $M_b(y)$.  Select
up to $80$ vectors with the most negative eigenvalues across all
blocks, requiring each selected eigenvalue to be below $-10^{-7}$.
\item \textbf{Scan coefficients and save.}
Evaluate $c_G(\mathcal{Q},\tau)$ for every $G\in\F^0_7$ and compute
$U=\max_{G\in\F^0_7}c_G(\mathcal{Q},\tau)$.  Save $(\mathcal{Q},\tau)$ if $U$
is the smallest value found so far within the search budget.
For each $H\in\F^0_6$, retain up to $8$ tetrahedron-free one-vertex
extensions with the largest coefficients; select as candidates
those with $c_G(\mathcal{Q},\tau)>B+10^{-7}$, allowing candidates
already in $W$.
\item \textbf{Apply cutting planes or column generation.}
Among the candidates selected in step~3, add those $G\notin W$
to $W$ (cutting planes).  If no new constraint is added,
add the vectors selected in step~2
to the corresponding sets $V_b$ (column generation).
Periodically remove vectors with negligible weights $\lambda_{b,v}$
from $V_b$ to control the LP size.
\end{enumerate}
Stop when the LP feasibility checks pass, $U-B\le10^{-7}$, and every
$M_b(y)$ has smallest eigenvalue at least $-10^{-7}$, or when the
search budget is exhausted.
Return the saved pair $(\mathcal{Q},\tau)$ with the smallest $U$.
\end{minipage}}
\caption{The optimization procedure used in the final search,
combining cutting planes and column generation.}
\label{fig:pipeline}
\end{figure}

\subsection{The dual SDP and LP}\label{app:dual-formulations}

The column-generation step in \cref{sec:implicit} uses the dual
multipliers $y_G$ to form $M_b(y)$ and select new vectors.
Our implementation computes these multipliers by solving the dual
of the restricted LP \eqref{eq:restricted-certificate} directly.
To explain the matrix constraints underlying this LP, we first
give the dual of the full certificate SDP, then describe the LP
solved at each iteration and how its solution supplies both $y$
and the certificate variables.

The certificate SDP \eqref{eq:sdp} has the following dual SDP,
written in terms of nonnegative weights $y=(y_G)_{G\in\F^0_7}$:
\begin{equation}\label{eq:dual-sdp}
\begin{array}{ll}
\text{maximize} & \displaystyle\sum_G y_G\,|E(G)|/35\\[1mm]
\text{subject to} & y_G\ge0,\quad\displaystyle\sum_Gy_G=1,\\[1mm]
 & \displaystyle\sum_Gy_Gs(G)=0,\qquad M_b(y)\PSD\quad(\text{every }b),
\end{array}
\end{equation}
where $M_b(y)=\sum_Gy_GA_b(G)$ is the weighted average of the
coefficient matrices for block $b$.
The PSD constraints apply to these averages; individual coefficient
matrices $A_b(G)$ need not be PSD.  Any feasible $y$ gives a lower
bound on the optimum of \eqref{eq:sdp} by weak duality: for every
feasible certificate $(B,\mathcal{Q},\tau)$,
\[
 \sum_G y_G\frac{|E(G)|}{35}
 \le \sum_G y_G\frac{|E(G)|}{35}+\sum_b\langle Q_b,M_b(y)\rangle
 =\sum_G y_G c_G(\mathcal{Q},\tau)
 \le B.
\]

For the fixed sets $W$ and $V_b$ at one iteration, the dual of the
restricted certificate LP \eqref{eq:restricted-certificate}
maximizes the same density objective with $y$ supported on $W$.
It replaces each PSD condition $M_b(y)\PSD$ by linear inequalities
for the selected vectors $v\in V_b$, giving the constraints
\begin{equation}\label{eq:psdcut}
 y_G\ge0,\qquad \sum_{G\in W}y_G=1,\qquad
 \sum_{G\in W}y_Gs(G)=0,\qquad
 v^{\mathsf T}M_b(y)v\ge0\quad(v\in V_b,\ \text{every }b).
\end{equation}
We solve this dual LP with HiGHS, a linear programming solver,
and use its variable values as $y$.
Since these inequalities enforce nonnegativity only along the selected
vectors, $M_b(y)$ can still have a negative eigenvalue.  A corresponding
eigenvector supplies a new direction for the column-generation step.

We recover the certificate variables from the constraint multipliers
returned by the same LP solve.  After adjusting solver sign conventions,
the multipliers of $v^{\mathsf T}M_b(y)v\ge0$ give the nonnegative
weights $\lambda_{b,v}$, which determine $Q_b$ through
\eqref{eq:matrix-atoms}.  The multiplier of $\sum_Gy_G=1$ gives $B$,
and that of $\sum_Gy_Gs(G)=0$ gives $\tau$, with the sign convention
of $c_G(\mathcal{Q},\tau)$.
When the implemented dual LP is feasible and bounded,
linear-programming duality ensures that its optimal constraint
multipliers give an optimal solution to the certificate LP
\eqref{eq:restricted-certificate}, with the same objective value.
Thus, the same LP solve provides $(B,\mathcal{Q},\tau)$ for the
coefficient scan and $y$ for the eigenvalue calculations in
\cref{fig:pipeline}.

\subsection{Constructing the exact certificate}\label{app:exact-construction}

Recovering exact certificates from floating-point SDP solutions
is standard in computational flag-algebra proofs
\cite{BaberTalbot2011,FalgasRavryVaughan2013}.
After the continued search described in Appendix~\ref{app:long-searches}
passed all convergence tests, we converted its final pair
$(\mathcal{Q},\tau)$ into the exact certificate of \cref{sec:cert}
as follows.

To keep the certificate compact, we first used a singular value
decomposition (SVD) to reduce the number of factor vectors before rounding.
Each term $\lambda_{b,v}vv^{\mathsf T}$ in
\eqref{eq:matrix-atoms} has factor vector $\sqrt{\lambda_{b,v}}\,v$.
For each block, let $R_b$ be the matrix whose rows are the
transposes of these factor vectors.  Write
$Q_b^{\mathrm{search}}=R_b^{\mathsf T}R_b$ for the matrix represented
by the saved factor vectors.
The SVD of $R_b$, with singular values $s_{b,r}$ and corresponding
unit right singular vectors $z_{b,r}$, gives
\[
 Q_b^{\mathrm{search}}
 =\sum_r(s_{b,r}z_{b,r})(s_{b,r}z_{b,r})^{\mathsf T}.
\]
Thus, the vectors $s_{b,r}z_{b,r}$ form an orthogonal list of
factor vectors representing the same matrix.  We discarded factor
vectors for which $s_{b,r}$ was
at most $10^{-7}$ times the largest singular value in that block.

Next, we rounded each coordinate of every retained factor vector to the
nearest multiple of $1/M$, where $M=8\cdot10^6$.  Writing these
rounded vectors as $a_{b,r}/M$, with $a_{b,r}\in\mathbb Z^{d_b}$,
we formed
\[
 Q_b^{\mathrm{exact}}=\frac1{M^2}\sum_r a_{b,r}a_{b,r}^{\mathsf T}.
\]
These are the rational matrices specified in \cref{sec:cert};
their factor representations ensure positive semidefiniteness.
We also rounded the stationarity multiplier to the nearest
multiple of $M^{-2}$.
The certificates of the experiments in \cref{sec:experiments} and
Appendix~\ref{app:experiments} use $M=2\cdot10^6$.  For the final
certificate we used $M=8\cdot10^6$, a value within the overflow
bounds of the integer verifiers; it reduced the loss from rounding,
the difference between $B$ and the floating-point coefficient maximum
$U$ of the converged pair, from approximately $5.7\times10^{-7}$ to
$8.7\times10^{-8}$.

Discarding factor vectors and rounding can change the coefficient maximum,
so we recomputed the bound using the rational matrices and multiplier.
We evaluated $c_G$ for every $G\in\F^0_7$ in exact arithmetic and
obtained the maximum $B$ stated in \cref{prop:seven-certificate}.
The resulting inequalities $c_G\le B$, together with the positive
semidefiniteness of the rational matrices, complete the construction
of the exact certificate.

\FloatBarrier
\section{Experimental settings and additional results}\label{app:experiments}

This appendix gives the settings, verification procedures, and
additional measurements for the experiments in \cref{sec:experiments}.

\subsection{Experimental settings}\label{app:experiment-details}

\paragraph{Block families.}
The six blocks with $2\ell-s=6$ in \cref{tab:fullfamilies}
have dimensions $2,11,64,56,50,45$.
The experimental implementations also included three blocks with
$2\ell-s=5$: one with $(s,\ell)=(1,3)$ and two with $(s,\ell)=(3,4)$,
of dimensions $2,8,7$, respectively.
Thus, \texttt{M6} and \texttt{M7-lift6} each used $9$ blocks,
\texttt{M7-no5} used $12$, and \texttt{M7-all5} used $35$,
as reported in \cref{tab:hostmodels}.
In \texttt{M7-no5} and \texttt{M7-all5}, expanding flags expresses
the quadratic forms of the three smaller blocks within the larger
blocks of the same types.  We retained them to make these smaller
representations directly available during optimization; their matrices
are zero in the final certificate of \cref{sec:cert}.

\paragraph{Initialization.}
For methods using cutting planes, we initialized $W$ as in
Appendix~\ref{app:optimization-procedure}, using the SDP's expansion size $m$.
Methods using column generation likewise started with
$V_b=\varnothing$ for every block.
Each run started afresh, without matrices or vectors inherited
from earlier searches.

\paragraph{Stopping criteria.}
After each solve, the implementations evaluated the current
candidate over all $G\in\F^0_m$ to obtain $U$ and computed the
smallest eigenvalue of every $M_b(y)$.  Here $y$ is the returned
weight vector in the dual formulations of Appendix~\ref{app:dual-formulations},
supported on $W$; for methods without cutting planes,
$W=\F^0_m$.  All four optimization methods used the following
tests for convergence of the overall search:
\[
  U-\sum_{G\in W}y_G\frac{|E(G)|}{\binom m3}\le10^{-7},
  \qquad
  \lambda_{\min}(M_b(y))\ge-10^{-7}
  \quad\text{for every block }b.
\]
The first test compares the global coefficient maximum with the
objective value returned by the solver; for an exactly solved
restricted LP, this objective value equals $B$.
Successful termination also required normalization, stationarity,
and nonnegativity to satisfy
\[
  \left|\sum_{G\in W}y_G-1\right|\le10^{-7},
  \qquad
  \left|\sum_{G\in W}y_Gs(G)\right|\le10^{-7},
  \qquad
  y_G\ge-10^{-7}\quad(G\in W).
\]

\begin{table}[t]
\centering
\caption{Computational costs and final counts of retained coefficient
inequalities and vectors: medians of three one-hour runs per
SDP--method pair.  Times are summed over all executions of each
task within a run.
Solve time includes problem assembly for \texttt{SDP-FULL}
and \texttt{SDP-CUT}.  Eigenvalue time covers calculations for
$M_b(y)$ outside the solver; coefficient evaluation time covers
the evaluation of $c_G$ for every $G\in\F^0_7$.  The final
values $|W|$ and $\sum_b|V_b|$ count retained coefficient
inequalities and vectors (equivalently, variables
$\lambda_{b,v}$), respectively.  A dash indicates that column
generation is not used.}
\label{tab:computationalcosts}
\small
\begin{tabular}{@{}llrrrrr@{}}
\toprule
SDP & method & \shortstack[r]{Solve\\time (s)}
& \shortstack[r]{Eigenvalue\\time (s)}
& \shortstack[r]{Coefficient\\evaluation time (s)}
& $|W|$ & $\sum_b|V_b|$\\
\midrule
\texttt{M7-no5} & \texttt{SDP-FULL} & $3514.14$ & $0.00$ & $9.33$ & $1{,}295{,}600$ & --\\
 & \texttt{LP-CG} & $1208.31$ & $0.20$ & $230.89$ & $1{,}295{,}600$ & $229$\\
 & \texttt{SDP-CUT} & $3477.89$ & $0.12$ & $120.59$ & $6688$ & --\\
 & \texttt{LP-CUT-CG} & $725.10$ & $2.61$ & $2787.83$ & $2859$ & $2328$\\
\addlinespace
\texttt{M7-all5} & \texttt{SDP-FULL} & $3372.62$ & $0.55$ & $11.27$ & $1{,}295{,}600$ & --\\
 & \texttt{LP-CG} & $1578.13$ & $11.74$ & $365.33$ & $1{,}295{,}600$ & $383$\\
 & \texttt{SDP-CUT} & $3172.24$ & $3.50$ & $79.14$ & $11758$ & --\\
 & \texttt{LP-CUT-CG} & $384.92$ & $104.35$ & $2930.59$ & $5078$ & $2792$\\
\bottomrule
\end{tabular}
\end{table}

\paragraph{Execution environment.}
All trials ran sequentially on Windows~11 with an Intel Core
i7-14700K, $63.72$~GiB RAM, and a $40$~GiB process-memory cap.
We used HiGHS~1.15.1 as the LP solver and SCS~3.2.8 as the SDP solver.
Solver and linear-algebra computations used one thread, and
complete coefficient scans used two.
For each run, memory was measured by summing the resident set
sizes (RSS) of the search process and its child processes.
The maximum observed total is the peak memory usage reported in
\cref{sec:experiments}.

The environment used Python~3.12.4, NumPy~2.0.2, SciPy~1.13.1,
CVXPY~1.6.5, and OpenBLAS~0.3.27; the C++ routines were compiled
with MinGW GCC~13.2.0.
For the seven-vertex \texttt{LP-CG} and \texttt{SDP-FULL} runs,
we precomputed the deletion probabilities $p(H,G)$ from
\cref{sec:evaluation} and the entries of $A_b(G)$ for blocks
whose flag products use seven vertices, for all $G\in\F^0_7$.
These data are independent of $(\mathcal Q,\tau)$ and were stored
as sparse arrays shared by runs of the same SDP.
They occupied $17.93$~GB on disk for \texttt{M7-no5}
and $20.42$~GB for \texttt{M7-all5}.  Constructing them took
$110.20$ and $132.33$ seconds, respectively; these setup costs
are excluded from the search timings.
The research archive retains the source versions, settings, flag
bases, per-run measurements, exact fractions, and verification
records; \nameref{sec:reproduction} gives the reproduction instructions.

\paragraph{Solver settings.}
For each individual solve, HiGHS used primal and dual feasibility
tolerances of $10^{-9}$, and SCS used absolute, relative, and
infeasibility tolerances of $10^{-8}$.
These internal solver settings are distinct from the $10^{-7}$
tests used above to decide whether the overall search had converged.
For \texttt{SDP-FULL} on six vertices, and for all \texttt{SDP-CUT}
runs, we used the SDP solver SCS through CVXPY, a Python package
for formulating convex optimization problems \cite{DiamondBoyd2016}.
The seven-vertex \texttt{SDP-FULL} runs used a validated SCS interface
that avoided copies of the large coefficient arrays.  The main \texttt{M6}
comparisons used direct linear-system solves with a limit of
$10^6$ iterations; the longer \texttt{M6} study and its separate
one-hour reference increased this limit to $20$ million.
Full seven-vertex SDP runs used indirect linear-system solves
with the same $20$-million limit and splitting parameter
$\rho_x=100$, fixed by preliminary trials.  Generated SDP runs
used $\rho_x=10^{-6}$ and a $300$-second limit per restricted
solve.  Settings were fixed within each group of three
repetitions.

\paragraph{Memory management.}
The \texttt{LP-CG}
implementations retained all coefficient inequalities and removed
vectors with negligible weights to limit memory.  Removal was
triggered when an update would exceed $500$ million entries in
the LP constraint matrix for \texttt{M7-all5}.
For \texttt{M7-no5}, the threshold was $300$ million entries.
Vectors with weights above
$10^{-10}$ were retained, and the cost of rebuilding the LP in
HiGHS after removal was charged to the search.  These memory
policies differ from those of the smaller LPs in
\texttt{LP-CUT-CG} and are part of the comparison in
\cref{sec:algorithmcomparison}.

\paragraph{Exact verification.}
Only candidate pairs $(\mathcal Q,\tau)$ whose complete coefficient
evaluation finished within the search budget were eligible for
selection.  Among these, we selected the pair with the smallest $U$.
We then applied the conversion procedure of
Appendix~\ref{app:exact-construction}, using the same denominators
and SVD tolerance for all experimental results.
For the SDP methods, negative eigenvalues
of exported certificate matrices were first replaced by zero;
the resulting factor vectors were then rounded and the largest
coefficient recomputed exactly over the full set of tetrahedron-free
$3$-graphs indexing the corresponding SDP.  This makes the reported upper
bounds valid even when the SDP solver had not reached a feasible
solution to the dual SDP in
Appendix~\ref{app:dual-formulations}.
Exact conversion and verification were performed after each search
ended and were excluded from the search budget.

\subsection{Computational costs}\label{app:experiment-results}

\Cref{tab:computationalcosts} compares the main computational costs
and the final counts of retained coefficient inequalities $|W|$
and vectors $\sum_b|V_b|$ for the four methods on the two
seven-vertex SDPs.  \texttt{SDP-FULL} and \texttt{SDP-CUT} spent
most of their running time in solver calls.  \texttt{LP-CUT-CG} retained only $2859$ and $5078$ coefficient
inequalities on \texttt{M7-no5} and \texttt{M7-all5}, respectively,
compared with $1{,}295{,}600$ for methods without cutting planes.
The dominant cost of \texttt{LP-CUT-CG} was instead evaluating
$c_G(\mathcal{Q},\tau)$ for every $G\in\F^0_7$ after each LP solve,
which computes $U$ and identifies violated constraints.
This step accounted for approximately $77\%$ of recorded running
time on \texttt{M7-no5} and $81\%$ on \texttt{M7-all5}, even with
the reductions of \cref{sec:evaluation}.  Exact conversion and verification took
approximately $74$--$113$ seconds per final certificate in these
one-hour comparisons.

\subsection{Longer searches}\label{app:long-searches}

To examine how much further the bounds could improve, we
increased the search budget to six hours and ran fresh searches
using \texttt{LP-CUT-CG} on \texttt{M7-no5} and \texttt{M7-all5}.
Both improved on the bounds from the one-hour comparisons in
\cref{sec:algorithmcomparison}, as shown in \cref{tab:longconvergence}.
\texttt{M7-no5} reached approximately $0.558917$, while
\texttt{M7-all5} reached approximately $0.557807$.
Both runs were stopped by their time limits before the convergence
tests were met.  Continuing the \texttt{M7-all5} search to
convergence, as described below, produced the bound in \cref{thm:main}.

\begin{table}[t]
\centering
\caption{Comparison of the one-hour bounds from
\cref{tab:m6methods,tab:sevenmethods} with bounds from fresh
six-hour searches.  All bounds were verified exactly and are
rounded to decimals.  Total factor counts are $\sum_b k_b$ over
all blocks of the final six-hour certificates.
All six-hour runs were time-limited.}
\label{tab:longconvergence}
\small
\begin{tabular}{@{}llrrr@{}}
\toprule
SDP & method & one-hour bound & six-hour bound & total factor count\\
\midrule
\texttt{M7-no5} & \texttt{LP-CUT-CG} & $0.559209644128$ & $0.558917128212$ & $437$\\
\texttt{M7-all5} & \texttt{LP-CUT-CG} & $0.558581395366$ & $0.557807255160$ & $794$\\
\texttt{M6} & \texttt{SDP-FULL} & $0.561665585236$ & $0.561653217956$ & $80$\\
\bottomrule
\end{tabular}
\end{table}

We also ran \texttt{SDP-FULL} on \texttt{M6} for six hours.
It reached approximately $0.561653$, a much smaller improvement
over its one-hour bound in
\cref{tab:m6methods}.
\Cref{fig:longconvergencemeasured} shows the progress of the longer
searches.  All three exhausted their time budgets without
establishing the SDP optimum.

Two settings also differed from the one-hour comparisons.
The six-hour \texttt{M7-all5} run used the procedure of
Appendix~\ref{app:optimization-procedure}, which allows column
generation whenever no new constraint is added.  The one-hour
runs performed this step only when $U-B\le10^{-7}$.
For \texttt{M6}, the SDP solver's iteration limit was increased
from one million to $20$ million.

\paragraph{Continuation to convergence.}
To obtain a certificate satisfying the convergence tests of
Appendix~\ref{app:experiment-details}, we continued the six-hour
\texttt{M7-all5} search from its final state ($6339$ retained
coefficient inequalities and $2322$ vectors) without a time limit,
using the procedure of Appendix~\ref{app:optimization-procedure}.
Two settings differed from the six-hour run.
First, the three blocks with $2\ell_b-s_b=5$, whose matrices were
zero in the six-hour certificate, were excluded.
Second, the threshold below which HiGHS discards small matrix
coefficients was lowered from its default of $10^{-9}$ to $10^{-12}$.

The continued search met all convergence tests of
Appendix~\ref{app:experiment-details} after $493$ further iterations
and $7.7$ hours of search time; at termination, $U-B=5.6\times10^{-10}$
and every $M_b(y)$ had smallest eigenvalue at least
$-9.7\times10^{-8}$.  Converting the final pair
$(\mathcal Q,\tau)$ as in Appendix~\ref{app:exact-construction},
with $M=8\cdot10^6$, gave the certificate of \cref{sec:cert}.

\begin{figure}[t]
\centering
\includegraphics[width=\textwidth,trim=0 6bp 0 0,clip]{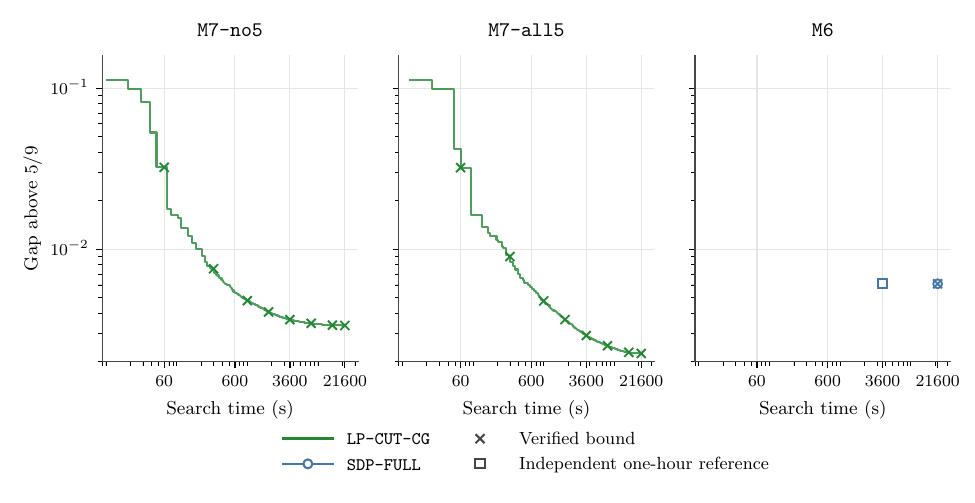}
\caption{Progress of the fresh six-hour searches on common logarithmic
axes.  Curves show the
smallest floating-point value of $U$ found so far; crosses mark
exact certificate bounds at the indicated search times, excluding
verification time.  Open circles mark the first available
\texttt{SDP-FULL} candidate.  The \texttt{M6} square marks an
additional one-hour reference run with a $20$-million solver
iteration limit, separate from the one-hour results in
\cref{tab:longconvergence}.  The continuation of the
\texttt{M7-all5} run to convergence is not shown.}
\label{fig:longconvergencemeasured}
\end{figure}

\end{document}